%% file: log.tex
\documentclass{article}
\pdfoutput=1
\usepackage{graphicx}
\usepackage{latexsym,amsmath, amsfonts,amscd, amsthm, amssymb}
\usepackage{epsfig}
\usepackage{changebar}
\usepackage{color}
\usepackage{bm}
\usepackage{tikz}
\usetikzlibrary{arrows,backgrounds,snakes}

\newtheoremstyle{plainNoItalics}{}{}{\normalfont}{}{\bfseries}{.}{ }{}

\theoremstyle{plain}
\newtheorem{thm}{Theorem}[section]

\theoremstyle{plainNoItalics}

\newtheorem{defn}[thm]{Definition}
\newtheorem{rem}[thm]{Remark}
\newtheorem{prop}[thm]{Proposition}
\newtheorem{exa}[thm]{Example}

\newcommand{\beq}{\begin{equation}}
\newcommand{\eeq}{\end{equation}}
\newcommand{\beqa}{\begin{eqnarray}}
\newcommand{\eeqa}{\end{eqnarray}}
\newcommand{\bit}{\begin{itemize}}
\newcommand{\eit}{\end{itemize}}
\newcommand{\bedef}{\begin{defn}}
\newcommand{\edefn}{\end{defn}}
\newcommand{\bpro}{\begin{prop}}
\newcommand{\epro}{\end{prop}}

\newcommand{\eps}{\varepsilon}

\newcommand{\bec}{\begin{array}{c}}
\newcommand{\ec}{\end{array}}

\renewcommand{\vec}[1]{{\bm #1}}

\def\Box{\mbox{ }\rule[0pt]{1.5ex}{1.5ex}}

\begin{document}



\input{title}

\newpage

\input{intro}

\input{framework}

\input{main_results}

\input{numerical_evidence}
\input{conclusion}
\input{acknowledgments}

\bibliographystyle{siam}
\bibliography{refer}

\end{document}

%% file: title.tex
\begin{center}
{\bf
Implicit-Explicit Integral Deferred Correction Methods for Stiff Problems 
}
\end{center}
\vspace{.2in}
\centerline{  
Sebastiano Boscarino \footnote{Department of Mathematics and Computer Science, University of Catania, Catania, 95127, E-mail: boscarino@dmi.unict.it}, 
Jing-Mei
Qiu\footnote{Department of Mathematics, University of Houston,
Houston, 77004. E-mail: jingqiu@math.uh.edu. Research supported by
Air Force Office of Scientific Computing grant FA9550-16-1-0179, NSF grant DMS-1522777 and University of Houston.},
Giovanni Russo \footnote{Department of Mathematics and Computer Science, University of Catania, Catania, 95125, E-mail: russo@dmi.unict.it} 
}

\bigskip
\noindent
{\bf Abstract.}

{The main goal of this paper is to investigate the order reduction phenomenon that appears in the integral deferred correction (InDC) methods based on implicit-explicit (IMEX) Runge-Kutta (R-K) schemes when applied to a class of stiff problems characterized by a small positive parameter $\eps$, called singular perturbation problems (SPPs).
In particular, an error analysis is presented for these implicit-explicit InDC (InDC-IMEX) methods when applied to SPPs.
In our error estimate, we expand the global error in powers of $\eps$ and show that its coefficients are global errors of the corresponding method applied to a sequence of differential algebraic systems. 
A study of these errors in the expansion yields error bounds 
and it reveals the phenomenon of
order reduction.  
In our analysis we assume uniform quadrature nodes excluding the left-most point in the InDC method and the globally stiffly accurate property for the IMEX R-K scheme. 
Numerical results for the Van der Pol equation and PDE applications are presented to illustrate our theoretical findings.


\bigskip
\noindent {\bf AMS Subject Classification Indices:}
Stiff equations 65L04, Multistep, Runge-Kutta and extrapolation methods 65L06, Extrapolation to the limit, deferred corrections 65B05, Methods for differential-algebraic equations 65L80.
\vfill

%% file: intro.tex
%

\section{Introduction}
\label{sec1}
\setcounter{equation}{0}
\setcounter{figure}{0}
\setcounter{table}{0}

Several physical phenomena of great relevance for applications are described by {autonomous} stiff systems of ordinary differential equations (ODEs) of the form
\begin{eqnarray}\label{start}
U' = F(U) + \frac{1}{\varepsilon}G(U), \ \ \ U(t_0) = U_0, 
\end{eqnarray}
where $F, G: \mathbb{R}^n \rightarrow \mathbb{R}^n$ are sufficiently smooth functions with different stiff properties and $\varepsilon$ is the stiffness parameters.
Usually, system (\ref{start}) with a large numbers of equations may arise from spatial discretization of a system of partial differential equations, such as convection-diffusion problems, diffusion-reaction ones  and hyperbolic systems with relaxation, (\cite{ascher1997implicit}, \cite{boscarino2010class}, \cite{pareschi2005implicit}, \cite{CarpenterKennedy}, \cite{liu1987hyperbolic}, \cite{zhong1996additive}), where the method of lines approach is usually used.

In order to be able to treat problems of the form (\ref{start}), it could be interesting to separate the non-stiff and the stiff
terms. In most cases $F(U)$ is non-linear and non-stiff and $G(U)$ contains the stiffness ${\varepsilon}$. Then it is desirable to develop
numerical methods which are explicit in $F$ and implicit in $G$. 
For the numerical integration of (\ref{start}), additive Runge-Kutta (R-K) methods are proposed, see for instance \cite{CarpenterKennedy, pareschi2005implicit, boscarino2010class, zhong1996additive}, and are chosen with the aim of efficiently integrating the system (\ref{start}).
Additive R-K methods combining implicit and explicit methods are also known in the literature as implicit-explicit (IMEX) R-K methods \cite{ascher1997implicit, CarpenterKennedy, pareschi2005implicit}.

We observe that system (\ref{start}) can be written as a system of $2n$ equations in the form 
\begin{eqnarray}\label{spp}
\begin{array}{l}
y' = f(y,z), \ \ \ y(t_0) = y_0,\\
\varepsilon z' = g(y,z), \ \ \ z(t_0) = z_0, 
\end{array}
\end{eqnarray} 
once we set $U = y + z$, $F(U) = f (y,z)$ and $G(U) = g(y, z)$. 
When splitting the system, one has to assign initial conditions for both $y$ and $z$, while only initial conditions on the sum are specified in the problem. 
Any choice of the initial value for $y$ and $z$, such $y_0+z_0=U_0$ will give the same dynamics for the sum $y+z$. Note however that such a splitting is adopted here for the purpose of analyzing the scheme. In practice, once the scheme is constructed, it is applied directly to systems in the form (\ref{start}), and no arbitrariness in the initial condition remains. 

System (\ref{spp}) is known in the literature as \emph{singular perturbetion problems} (SPPs). Classical books on SPPs are \cite{Tikhonov, Malley} and we also refer the reader to the book \cite{hairer1993solving2}. 
System (\ref{spp}) has several origins in applied mathematics and allows us to understand many phenomena observed for very stiff problems, for example {as mentioned in \cite{hairer1993solving2}: from fluid dynamics and results in linear boundary value problems containing  a small parameter $\eps$ (the viscosity coefficient), in the study of stiff nonlinear oscillators (such as the Van der Pol system with large value of the parameter) or in the study of chemical kinetics with slow and fast reactions.} In fact, when the parameter $\varepsilon$ is small, the corresponding differential equation is stiff, and when $\varepsilon$ tends to zero, the differential equation becomes
differential algebraic. A sequence of differential algebraic systems arises in the study of SPPs. {In \cite{hairer1993solving2} and in the original paper \cite{hairer1988error}, the authors,} {by studying these differential algebraic systems, offered a general framework in order to derive rigorous convergence estimates of most of the implicit R-K methods presented in the literature when applied to SPPs (\ref{spp}) and showed that these methods suffer from the phenomenon of order reduction in the stiff regime.}
 
 {Due to the equivalence of the two systems \eqref{start} and \eqref{spp}, the error estimate for \eqref{spp} offers a theoretical foundation to understand the error behavior of \eqref{start}. 
In this paper, we perform our error estimate base on the system \eqref{spp}, while the application problems we considered are often in the form of \eqref{start}.}

The main goal in this paper is to investigate the order reduction phenomenon that appears in the InDC framework when combined with IMEX R-K methods (InDC-IMEX) when applied to SPPs. The novelty here is to provide rigorous and careful convergence analysis for the global error of InDC-IMEX method. {As far as we know, InDC methods constructed using IMEX R-K methods, have not been seriously studied in the literature when applied to SPPs. Furthermore, the error estimates of InDC-IMEX methods given in this paper for SPPs, explain theoretically, for the first time, the order reduction phenomenon that usually occurs for mildly stiff or stiff problems.}

We recall  that the InDC method \cite{christlieb2009integral, christlieb2009comments}, along the development of deferred correction (DC) \cite{CDC, SRD} and spectral deferred correction (SDC) \cite{dutt2000spectral, minion2003semi, layton2007implications, layton2008choice, layton2005implications, huang2006accelerating} methods, is an automatic technique of building up very high order numerical integrators based on lower order ones for ODEs. The procedure consists of one prediction and iterations of correction steps. The high order accuracy is accomplished by using a lower order numerical method to solve a series of error equations in each correction step. Compared with the classical DC method, the recently developed SDC and InDC methods are based on Picard integral equation and a deferred correction procedure is applied to an integral formulation of the error equation in DC methods. It has been shown that SDC and InDC outperforms DC in many problems with better stability and accuracy properties \cite{dutt2000spectral, christlieb2009integral}. The main difference between SDC and InDC is the distribution of quadrature nodes: the SDC method uses Gaussian/Lobatto/Radau points for better stability and accuracy properties, while InDC method uses uniform quadrature points to guarantee high order accuracy increase when high order R-K methods are applied in correction steps \cite{christlieb2009integral}. {The InDC method can be considered as a R-K type method by assembling the corresponding Butcher table, see \cite{christlieb2009comments} and Section~\ref{InDC_SA} later on in this paper. In fact, R-K or additive R-K methods of third order or less have been very well developed and optimized for their efficiency and effectiveness \cite{CarpenterKennedy}. The new aspect, that the InDC method offers, is a systematical way in constructing very high order method (e.g. up to 12th order \cite{idcark}) without the need to working out order conditions for both explicit and implicit methods. Moreover, the InDC method can be used to improve the operator splitting error \cite{christlieb2014high}, which is not possible by using a R-K method.}

The development of this theoretical study of IMEX InDC method is aided by the knowledge of some results obtained in the following papers: \cite{boscarino2008error} and \cite{boscarino_qiu}. In \cite{boscarino2008error}, the author studied the global error behavior of IMEX R-K methods existing in the literature, presenting convergence proofs for different types of IMEX R-K methods, and gave error bounds for such
methods when applied to the SPP system (\ref{spp}). In particular, this study revealed that IMEX R-K methods suffer from the phenomenon
of order reduction in the stiff regime, i,e. $\varepsilon \to 0$. In a similar fashion, in {\cite{boscarino_qiu}} the authors studied the global error behaviour of InDC method constructed by using implicit R-K methods when applied to the SPP system (\ref{spp}). This study gave error bounds for such methods. In particular, it revealed the phenomenon of order reduction in the stiff regime, when we increase the order of accuracy by series of correction equations.

This paper is a natural continuation of the research started in \cite{boscarino_qiu}. 
Then by combining the results presented in \cite{boscarino2008error} and \cite{boscarino_qiu}, we study the error bounds of InDC-IMEX methods for SPPs (\ref{spp}) in order to seek an understanding on the order reduction phenomenon. Notice that the main idea is to expand the error in powers of $\varepsilon$ whose coefficients are called error terms, and show convergence results for these error terms, as done in  \cite{boscarino2008error}. {In our analysis we consider $h \gg \varepsilon$, with $h$ being the time step size.} We consider some suitable assumption for our analysis: \emph{the globally stiff accuracy} (GSA) for the IMEX R-K methods (we will define it in the next), and the use of uniform quadrature nodes excluding the leftmost endpoint in the InDC framework. The first assumption guarantees that if a high order globally stiffly accurate IMEX R-K method is used to construct the InDC IMEX R-K method, the assembled matrix is invertible and by this result, the order of convergence for the leading order term in the $\varepsilon$-expansion of global error increases with high order in the correction steps. 
Furthermore, we will show by a counter example that, if the assumption of GSA for the IMEX R-K method is not satisfied, the corresponding InDC IMEX R-K method does not have the order increase as expected for the leading order term in the $\varepsilon$-expansion of the global error. {Note that the assumption of GSA provides a sufficient conditions to guarantee the invertibility of assembled matrix. In general we do not know whether GSA is also a necessary condition. To find this out requires a more detailed investigation.}

{Finally, note that the uniform distribution of nodes is important to increase accuracy when a high order R-K method is applied in the correction steps for classical problems and the use of quadrature nodes excluding the left-most endpoint leads to an important stability condition for stiff problems, i.e. the method becomes L-stable if A-stable; we refer readers to \cite{layton2005implications} and \cite{christlieb2009integral} for details. }


The paper is organized in the following way. In Section~\ref{sec:2}, we briefly present existing local and global error estimates of IMEX R-K methods for SPPs \cite{boscarino2008error}. In Section~\ref{sec:4}, we introduce InDC-IMEX methods applied to SPPs \eqref{spp}. {A reformulation of InDC methods constructed with IMEX R-K methods is obtained, with assembled matrices in double Bouchet tableau as in a classical IMEX method.} In Section~\ref{sec3}, main theoretical results are stated in the form of a theorem, which is proved via the classical error estimate of IMEX methods presented in \cite{boscarino2008error}. Numerical evidences, supporting these theoretical results, are summarized and presented in Section~\ref{sec5}. Conclusions are given in Section~\ref{sec6}.

%% file: framework.tex
\section{SPPs and IMEX R-K methods}
\label{sec:2}
\setcounter{equation}{0}
\setcounter{figure}{0}
\setcounter{table}{0}

In this section, we review classical concepts and results of the R-K and IMEX R-K methods when applied to SPPs (\ref{spp}).
In system (\ref{spp}) we assume that $0 < \varepsilon \ll 1$ and $f$ and $g$ are sufficiently differentiable vector-valued functions. The functions  $f$, $g$ and the initial values $y(0)$, $z(0)$ may depend smoothly on $\varepsilon$. For simplicity of notation, we suppress such dependence. We require that system (\ref{spp}) satisfies 
\beq
\label{eq: gz}
\mu(g_z(y, z)) \le -1,
\eeq
in an $\varepsilon$-independent neighbourhood of the solution, where $\mu$ denotes the logarithmic norm with respect to some inner product, see for details \cite{hairer1993solving2}. 

When $\varepsilon = 0$, the corresponding \emph{reduced} system is the differential algebraic equation (DAE) 
\begin{eqnarray}\label{reduced}
\begin{array}{l}
y' = f(y,z), \\
0 = g(y,z), 
\end{array}
\end{eqnarray}
whose initial values are \emph{consistent} if $0 = g(y(0),z(0))$. We assume that the Jacobian $g_z(y, z)$
 is invertible in a neighborhood of the solution of (\ref{reduced}). This
assumption guarantees the solvability of the second equation in (\ref{reduced}) and that the equation $g(y, z) = 0$ possesses
a locally unique solution $z = \mathcal{G}(y)$ (implicit function theorem). Then insert it  into the first equation of (\ref{reduced}), this gives
\begin{eqnarray}\label{eqy} 
y' = f(y,\mathcal{G}(y)).
\end{eqnarray}
The same assumption guarantees that system (\ref{reduced}) is 
a differential-algebraic one of index 1, (see \cite{hairer1993solving2} for more details).

From a classical result in SPPs theory, condition (\ref{eq: gz}) guarantees the existence of an $\varepsilon$-expansion, as the sum of a smooth function of the independent variable $t$ and an exponentially decaying function of the stretched variable $\tau = t/\varepsilon$ (initial layer).
The exponentially decaying function is not present if the initial values of system (\ref{spp}) (which depend on $\eps$) are on the smooth solution,  see Chap. VI.3 in \cite{hairer1993solving2} for more details. 

Thus in our analysis we seek mainly $\varepsilon$-asymptotic expansion of the exact smooth solutions of the problem (\ref{spp}) of the form
\begin{eqnarray}\label{ExpSol}
\begin{array}{c}
\displaystyle y(t) = y_0(t) + y_1(t)\eps + y_2(t)\eps^2 + \cdots, \\
\displaystyle z(t) = z_0(t) + z_1(t)\eps + z_2(t)\eps^2 + \cdots. 
\end{array}
\end{eqnarray}
{Inserting (\ref{ExpSol}) into (\ref{spp}) and collecting equal power of $\varepsilon$ we have, \cite{hairer1993solving2}:
\begin{equation}\label{g1}
g(y_0,z_0) = 0,
\end{equation}
\begin{equation}\label{g2}
y_0' = f(y_0,z_0),
\end{equation}
\beq
\label{g3}
\eps^1: \quad
\left \{
\begin{array}{l}
y'_1 = f_y (y_0, z_0)y_1+f_z (y_0, z_0)z_1, \\[2mm]
z'_0 = g_y (y_0, z_0)y_1+g_z (y_0, z_0)z_1, \\
\end{array}
\right.
\eeq
\beq
\cdots \nonumber
\eeq
\beq
\label{gnu}
\eps^{\nu}: \quad
\left \{
\begin{array}{ll}
y'_{\nu} &= f_y (y_0, z_0)y_{\nu}+f_z (y_0, z_0)z_{\nu} + \phi_{\nu} (y_0, z_0,\cdots, y_{\nu-1},z_{\nu-1}), \\[2mm]
z'_{\nu-1} &= g_y (y_0, z_0)y_{\nu}+g_z (y_0, z_0)z_{\nu} + \psi_{\nu} (y_0, z_0,\cdots, y_{\nu-1},z_{\nu-1}).\\
\end{array}
\right.
\eeq
%
Eq. (\ref{g1}) tells us that $z_0$ is algebraically related to $y_0$. Differentiating (\ref{g1}) we have 
\[
g_y(y_0,z_0)y_0' + g_z(y_0,z_0) z_0' = 0.
\]
Compatibility with (\ref{g2}) implies
\[
g_y(y_0,z_0)f(y_0,z_0) + g_z(y_0,z_0)(g_y(y_0,z_0) y_1 + g_z(y_0,z_0) z_1) = 0.
\]  
Therefore also $z_1$ is linked to $y_0$, $z_0$, and $y_1$, by an algebraic relation.
Then we have the following definition:
\begin{defn}\label{WP}
An initial condition $(y(0), z(0))$ for the system (\ref{spp}) is \emph{well-prepared} to order $\nu$ in $\varepsilon$ if the expansion 
\begin{eqnarray}\label{ExpSol0}
\begin{array}{c}
\displaystyle y(0) = y_0(0) + y_1(0)\eps + y_2(0)\eps^2 + \cdots, \\
\displaystyle z(0) = z_0(0) + z_1(0)\eps + z_2(0)\eps^2 + \cdots, 
\end{array}
\end{eqnarray}
satisfies the algebraic conditions (\ref{g1}), (\ref{g2}), (\ref{g3}) and (\ref{gnu}) up to order $\nu$ in $\varepsilon$.
\end{defn}
 Note that arbitrary initial values introduce an initial layer in the solution. One possible way to overcome this difficulty is simply to use a small time step $\Delta t = {\mathcal O}(\varepsilon)$ during the initial transient. After a short time, the initial layer is damped out, and the time step can be chosen larger than $\varepsilon$. All this is usually performed by a suitable time step control \cite{hairer1993solving2}.
Here we are interested in the behaviour of the scheme for time much larger that $\varepsilon$, after the effect of the initial layer has damped out, and this is why we assume that our initial condition is well-prepared. In practice, any initial condition chosen as the solution of the system at a given time large enough compared to $\varepsilon$ will be well prepared}.

A sequence of differential-algebraic systems arises in the study of SPPs, i.e. (\ref{g1}), (\ref{g2}), (\ref{g3}) and (\ref{gnu}), under the assumption of a smooth solutions (\ref{ExpSol}) of system (\ref{spp}) (a general and detailed investigation about the $\varepsilon$-expansion is given in \cite{hairer1993solving2}).

As we will show next, by inserting the ansatz (\ref{ExpSol}) into system (\ref{spp}) and collecting equal power of $\varepsilon$, one obtain a set of conditions where the coefficients in the expansion (\ref{ExpSol}) are the solutions of differential algebraic systems. 

We remind that in order to prove the convergence of R-K methods for problem of the form (\ref{spp}), in \cite{hairer1993solving2} and in the original paper \cite{hairer1988error}, the authors showed that most of the implicit R-K methods presented in the literature suffer from the phenomenon of the order reduction when applied to (\ref{spp}) in the stiff regime ($\varepsilon \to 0$). In \cite{boscarino2008error}, similar results of convergence are obtained for IMEX R-K methods.


Now we consider an $s$-stage IMEX R-K method with a double $tableau$ in the usual Butcher notation,
\begin{equation}\label{DBT}
\begin{array}{c|c}
\tilde{c} & \tilde{A}\\
\hline
\vspace{-0.25cm}
\\
 & \tilde{b^T} \end{array} \ \ \ \ \  
\begin{array}{c|c}
{c} & {A}\\
\hline
\vspace{-0.25cm}
\\
 & {b^T} \end{array}.
\end{equation}
where $\tilde{A} = (\tilde{a}_{ij})$ is an $s \times s$ matrix for an explicit scheme, with $\tilde{a}_{ij}=0$ for $j \geq i$ and $A = ({a}_{ij})$ is an  $s \times s$ matrix for an implicit scheme. 
For the implicit part of the methods, we use a diagonally implicit scheme for the function $g$, i.e. $a_{ij}=0$, for $j > i$, in order to guarantee simplicity and efficiency in solving the algebraic equations corresponding to the implicit part of the discretization. The vectors $\tilde{c}=(\tilde{c}_1,...,\tilde{c}_s)^T$, $\tilde{b}=(\tilde{b}_1,...,\tilde{b}_s)^T$, and $c=(c_1,...,c_s)^T$, $b=(b_1,...,b_s)^T$ complete the characterization of the scheme. The coefficients $\tilde{c}$ and $c$ are given by the usual relation
\begin{eqnarray}\label{eq:candc}
\tilde{c}_i = \sum_{j=1}^{i-1} \tilde a_{ij}, \ \ \ c_i = \sum_{j=1}^{i} a_{ij}.
\end{eqnarray}
When such IMEX scheme is applied to the SPP (\ref{spp}), we have
\begin{eqnarray}\label{eq:1-8}
\left(\begin{array}{c}
y_{n+1}\\
\varepsilon z_{n+1}
\end{array}\right)
 = \left(\begin{array}{c}
y_{n}\\
\varepsilon z_{n}
\end{array}\right)
 + h \sum_{i=1}^{s} \left(\begin{array}{c}
                          \tilde{b}_i k_{i}\\
                          b_i  \ell_{i}
                          \end{array}\right) 
\end{eqnarray}
where
\begin{eqnarray}\label{eq:1-9}
\left(\begin{array}{c}
k_{i}\\
\ell_{i}
\end{array}\right)
 = \left(\begin{array}{c}
f(Y_{i},Z_{i})\\
g(Y_{i},Z_{i})
\end{array}\right) 
\end{eqnarray} 
and the internal stages are given by
\begin{eqnarray}\label{eq:1-10}
\left(\begin{array}{c}
Y_{i}\\
\varepsilon Z_{i}
\end{array}\right)
 = \left(\begin{array}{c}
y_{n}\\
\varepsilon z_{n}
\end{array}\right)
 + h  \left(\begin{array}{c}
                         \sum_{j=1}^{i-1} \tilde{a}_{ij} k_{j}\\
                         \sum_{j=1}^{i} a_{ij}  \ell_{j}
                          \end{array}\right).
\end{eqnarray}
{The following definition will be also useful to characterize properties of an IMEX R-K method in the sequel.}
\begin{defn}
\label{GSA}
We say that an IMEX R-K scheme is \emph{globally stiffly accurate} (GSA) if $b^T = e^T_s
A$, 
and $\tilde{b}^T = e^T_s
\tilde{A}$, with $e_s = (0, . . . , 0, 1)^T$, and $c_s = \tilde{c}_s = 1$, i.e. the numerical solution is
identical to the last internal stage value of the scheme.
\end{defn}
{This definition means that the IMEX R-K scheme is a {\em stiffly accurate}, (i.e. ${a}_{si} = {b}_i, \ \ \textrm{for} \ \ i = 1,...,s$) in the implicit part \cite{hairer1993solving2} and a FSAL (\emph{First Same As Last}) R-K method  \cite{hairer1993solving1} in the explicit part, (i.e. 
 $\tilde{a}_{si} = \tilde{b}_i, \ \ \textrm{for} \ \ i = 1,...,s-1$)}.
It is known in the literature that FSAL R-K methods are a special class of $s$-stage explicit R-K
schemes. Such schemes have the advantage that they require only $s-1$ function evaluations per time step, because
the last stage of step $n$ coincides with the first step of the step $n+1$ (see \cite{hairer1993solving1} for details).
Note that IMEX R-K schemes with this property were already
introduced in \cite{boscarino2013flux, boscarino2013implicit}. Next we shall show the importance of this property 
for the convergence of the IMEX R-K method.

{In our analysis, we use $R(\infty) = \lim_{z\rightarrow -\infty} R(z)$, with $R(z)$ \emph{the stability function} of the implicit scheme, given by $R(z) = 1 + zb^{T}(I-zA)^{-1}${\bf 1}, with $b^{T} = (b_1,...,b_s)$ and  {\bf 1}$ = (1,...,1)^{T}$. From the expression of $R(z)$ we have $R(\infty)=1-b^TA^{-1}${\bf 1}. An implicit R-K method is said to be $L$-stable if it is $A$-stable and $R(\infty) = 0$. The $L$-stability property is important when one treats with stiff problem.(Chap. IV.3 in \cite{hairer1993solving2}). Furthermore, in the special case when the implicit method is stiffly accurate and the matrix $A$ is invertible, one has always $R(\infty) = 0$, and this makes an $A$-stable method $L$-stable.

Now we assume that the implicit R-K matrix $(a_{ij})$ is invertible, so that we get
\begin{eqnarray*}
h \ell_{ ni} = \varepsilon \sum_{j = 1}^{s}\omega_{ij}(Z_{nj}-z_n),
\end{eqnarray*}
where $\omega_{ij}$ are the elements of the inverse of $(a_{ij})$. Note that this assumption is important to provide the error estimates in \cite{boscarino2008error}.
Then inserting this into $z_{n+1}$, and setting $\varepsilon = 0$ we obtain
\begin{eqnarray}
  Y_{i} & = & y_n + h \sum_{j=1}^{i-1}\tilde{a}_{ij}f(Y_{j},Z_{j}) \label{eqg1} \\ 
 0 & = & g(Y_{i},Z_{i}) \label{eqg} \\
  y_{n+1} & = & y_n + h\sum_{i=1}^{s}\tilde{b}_if(Y_{i},Z_{i}) \label{eqg2} \\ 
 z_{n+1} & = & R(\infty)z_n + \sum_{i,j=1}^{s}b_i\omega_{ij}Z_{j}. \label{eqg3} 
\end{eqnarray}

In our analysis it is useful to characterize the different type of IMEX R-K methods existing in the literature we will consider in the sequel accordingly
to the structure of the implicit part of the method. Following \cite{boscarino2008error} we have
\begin{defn}\label{Def:A}
We call an IMEX R-K method of type A, if the matrix $A \in R^{s \times s}$ is invertible and $\tilde{c} \neq c$.
\end{defn} 
\begin{defn}\label{Def:CK} 
We call an IMEX R-K method of type CK, if the matrix $A \in R^{s\times s}$ can be written as
\begin{eqnarray*}
A = \left(\begin{array}{ll} 0 & 0\\
                        a  & \hat{A}\end{array}\right)
               \end{eqnarray*}         
with the submatrix $\hat{A} \in R^{(s-1)\times (s-1)}$ invertible and and $\tilde{c} = c$. In particular when $a = 0$, the IMEX R-K method is of type ARS.
\end{defn}
In \cite{boscarino2008error}, the author presented the error analysis of different types of IMEX R-K schemes when applied to SPP (\ref{spp}), some of which is summarized in Section~\ref{sec3} below.
Now briefly we review a couple of main results from \cite{boscarino2008error}, which we will use to prove the main theorem in this paper.

In order to obtain the error estimate for IMEX R-K methods, we start from the $\varepsilon$-expansion of the exact and numerical solutions of the problem (\ref{spp}), (see \cite{boscarino2008error} and \cite{hairer1993solving2} for details). Then, by assuming that the initial values are {well-prepared}, Definition \ref{WP}, an $\eps$-asymptotic expansion of smooth solutions of the system (\ref{spp}) is given by the ansatz (\ref{ExpSol}):
\beq
\label{eq: exact-esp}
y(t)=\sum_{\nu\ge 0} y_{\nu}(t) \eps^\nu, \quad
z(t)=\sum_{\nu\ge 0} z_{\nu} (t) \eps^\nu,
\eeq
and, similarly, for the numerical solutions by:
\beq
\label{eq: exp_expand}
y_n= \sum_{\nu \ge 0} y_{n, \nu} \eps^\nu, \qquad
z_n =\sum_{\nu \ge 0} z_{n, \nu} \eps^\nu,
\eeq
approximating exact solutions at $t_n$.
Then, the errors of the $y$ and $z$-component are formally considered as
\beq
\label{eq: Err_exp_expand_sol}
y(t_n) -{ y}_n= \sum_{\nu\ge0} \eps^\nu(y_{\nu}(t_n) - {y}_{n, \nu}) , \qquad
z(t_n) - {z}_n =\sum_{\nu\ge0} \eps^\nu(z_{\nu}(t_n) - {z}_{n, \nu}).
\eeq
The values $y_{\nu}(t)$, $z_{\nu}(t)$ in (\ref{eq: exact-esp}) are the coefficients of the $\eps$-expansion of the smooth solution for (\ref{spp}) and values $y_{n,0}$, $z_{n,0}$, $y_{n,1}$, $z_{n,1}$, $\cdots$, in (\ref{eq: exp_expand}) represent the numerical solutions of the IMEX R-K method applied to differential algebraic equations (DAEs) of arbitrary order $\nu$. Furthermore,
the first differences $y_0(t_n)- y_{n,0}$ and $z_0(t_n)- z_{n,0}$ in the expansion (\ref{eq: Err_exp_expand_sol}) are the global errors of IMEX R-K method
applied to the reduced system (\ref{reduced}), i.e. system of index $\nu = 1$. The other differences for $\nu>0$ in (\ref{eq: Err_exp_expand_sol}) are related
to the numerical solutions of the IMEX R-K method when applied to the DAEs of higher index.  In order to study the error (\ref{eq: Err_exp_expand_sol}), in \cite{boscarino2008error} the author investigated the differences: $y_{\nu}(t_n) - y^n_{\nu}$, $z_{\nu}(t_n) - z^n_{\nu}$, for $\nu \ge 0$. For details, see \cite{boscarino2008error}.

{Finally for the next analysis a couple of remarks are in order for GSA IMEX R-K of a given type.}
\begin{rem}
\label{rem: index1}
a)  By the Implicit Function Theorem applied to (\ref{eqg}), we have $Z_{i}  =  \mathcal{G}(Y_{i})$ for $i=1,...,s$
and the internal stages $Z_{i}$ depend on the internal stages of the explicit part of $Y_{i}$. 
Furthermore, in general, the numerical solutions $y_{n+1}$, $z_{n+1}$ do not lie on the manifold $g(y,z) =0$. 
However, if the scheme is GSA, then $y_{n+1} = Y_{ns}$, $z_{n+1} = Z_{ns}$ and therefore equation 
\begin{eqnarray}\label{eqsol}
g(y_{n+1},z_{n+1}) = 0,
\end{eqnarray}
is satisfied anyway.

b) In this case the application of (\ref{eqg1})-(\ref{eqg})-(\ref{eqg2}) and (\ref{eqsol}) to system (\ref{reduced}) is equivalent to the $s$-stage explicit R-K method (\ref{eqg1})-(\ref{eqg2})  applied to (\ref{eqy}) when $\eps = 0$.
Then by (\ref{eqsol}), we get $z_{n+1} = \mathcal{G}(y^{n+1})$ and we obtain for the $y$ and $z$ component the following estimates:
\begin{eqnarray*}
y_n - y(t_n)  =  \mathcal{O}({h^p}), \quad z_n - z(t_n)  = \mathcal{O}({h^p}),
\end{eqnarray*}
with $p$ is  the classical order of the explicit R-K method. Then the $z$-component possesses the same asymptotic error estimate as the $y$-component in the case $\varepsilon = 0$. Note  that if the scheme is not GSA, the equation (\ref{eqsol}) is not satisfied and  a loss of accuracy is observed for the variable $z$ (for details, see \cite{boscarino2008error}).

c)  The globally stiffly accurate property  for an IMEX R-K scheme implies also that $\tilde{b}_i \neq  b_i$, for $i = 1,\cdots, s$. Note that in \cite{boscarino2008error} the assumption of  $\tilde{b}_i = {b}_i$ for all $i$ for the IMEX R-K methods  represents the only remedy for preserving the order for the differential component $y$  when we consider high index DAEs, i.e. $\nu > 0$, otherwise the order drops to first one (for details, see Theorem 5.2, Theorem 6.1 and Theorem 6.2 in \cite{boscarino2008error}). 
\end{rem}

\section{InDC-IMEX R-K formulations applied to SPPs}
\label{sec:4} 
\setcounter{equation}{0}
\setcounter{figure}{0}
\setcounter{table}{0}

In this section, we consider the InDC methods constructed using IMEX schemes applied to SPPs (\ref{spp}). {First of all, we introduce the specific strategy of InDC methods. Later, we show as a InDC IMEX R-K method can be re-written as a standard IMEX R-K one. Clearly this formulation provides a systematic way to construct arbitrary-order IMEX R-K schemes without solving complicated order conditions. The formulation of InDC IMEX R-K methods as a standard IMEX R-K gives us the possibility  to estimate the error of the schemes by using some classical results of the error analysis developed in the paper \cite{boscarino2008error} for standard IMEX R-K.}
\subsection{{InDC framework}}

The time interval {$[0, T]$ is uniformly discretized into intervals $[t_n, t_{n+1}]$}, $n = 0,1,...,N-1$ such that
\begin{eqnarray*} 
0 = t_0 < t_1 < t_2 < ... < t_n < ...< t_N = T,
\end{eqnarray*}
with the step size $H$. Each interval $[t_n, t_{n+1}]$ is discretized again into $M$ uniform subintervals with quadrature nodes denoted by 
\begin{eqnarray}
\label{eq: node} 
t_n \doteq \tau_0 < \tau_1 <\cdots < \tau_M \doteq t_{n+1},
\end{eqnarray}
with $h = \frac{H}{M}$ being the size of a substep. In this paper, the interval $[t_n,  t_{n+1}]$ will be referred
to as a time step while a subinterval $[\tau_m, \tau_{m+1}]$ will be referred to as a substep. 
We assume the InDC quadrature nodes are uniform, which is a crucial assumption for high order improvement in accuracy, when consider applying general high order R-K methods in prediction and correction steps for classical ODE system, (see discussions in \cite{christlieb2009integral}). Similar conclusions hold when the high order IMEX R-K methods are applied in prediction and correction steps. {Since $H = M h$, we will use $\mathcal{O}(h^p)$ and $\mathcal{O}(H^p)$ interchangeably throughout the paper.  The use of quadrature nodes excluding the left-most endpoint (i.e. $\tau_0$) leads to an important stability condition for stiff problems, i.e. the method is L-stable 
{(for details see in \cite{layton2005implications})}. With the above considerations, in this paper, we consider the InDC methods with uniform quadrature nodes excluding the left-most endpoint.}

Let's assume we have obtained numerical solutions $\hat{y}^{(0)}_m$ and $\hat{z}^{(0)}_m$ approximating the exact solution at $\tau_m$ by using a low order numerical method for (\ref{spp}). Here superscript $(0)$ is used to denote the prediction step in the InDC method. We build continuous polynomial interpolants $\hat{y}^{(0)}(t)$ and $\hat{z}^{(0)}(t)$ interpolating these discrete values ($\hat{y}^{(0)}_m$ and $\hat{z}^{(0)}_m$ for $m=1, \cdots M$). Now we define the error functions 
\beq
\label{eq: error_function}
e^{(0)}(t) = y(t)-\hat{y}^{(0)}(t), \quad d^{(0)}(t) = z(t)-\hat{z}^{(0)}(t).
\eeq
Note that $e^{(0)}(t)$ and $d^{(0)}(t)$ are not polynomials in general.
We specify the residual function with respect to $y$ and $z$ via the following set of differential equations
\beq
\label{eq: delta_diff}
\begin{array}{l}
\delta^{(0)}(t) = f(\hat{y}^{(0)}(t),\hat{z}^{(0)}(t))-(\hat{y}^{(0)})'(t), \\[2mm]
\rho^{(0)}(t) = g(\hat{y}^{(0)}(t),\hat{z}^{(0)}(t))-(\varepsilon \hat{z}^{(0)})'(t). 
\end{array}
\eeq
{Thus, by subtracting \eqref{eq: delta_diff} from \eqref{spp}, {and recalling that we restrict our analysis to autonomous systems,} the error equations about the error functions \eqref{eq: error_function} become
\beq
\label{defint}
\begin{array}{l}
(e^{(0)})'(t) - \delta^{(0)}(t)= \Delta f^{(0)},\\[2mm] 
\varepsilon (d^{(0)})'(t)-\rho^{(0)}(t) = \Delta g^{(0)},
\end{array}
\eeq
where
\beq
\label{eq: def_deltafg}
\begin{array}{ll}
\Delta f^{(0)} &\doteq f(e^{(0)}(t)+\hat{y}^0(t),d^{(0)}(t)+\hat{z}^{(0)}(t))-f(\hat{y}^{(0)}(t),\hat{z}^{(0)}(t)),\\
\Delta g^{(0)} &\doteq g(e^{(0)}(t)+\hat{y}^0(t),d^{(0)}(t)+\hat{z}^{(0)}(t))-g(\hat{y}^{(0)}(t),\hat{z}^{(0)}(t)).
\end{array}
\eeq
{The initial conditions for these error equations are zero, i.e. $e^{(0)}(0) = 0$ and $d^{(0)}(0)=0$.}
}

Suppose that we have obtained approximate solutions $\hat{e}^{(0)}_m$ and $\hat{d}^{(0)}_m$ at $\tau_m$ by using a low order IMEX method for error equations \eqref{defint}, the numerical solution can then be improved as 
\[
\hat{y}^{(1)}_m = \hat{y}^{(0)}_m +  \hat{e}^{(0)}_m, \quad \hat{z}^{(1)}_m = \hat{z}^{(0)}_m +  \hat{d}^{(0)}_m, \quad \forall m = 0, \cdots M.
\]
Such correction procedures can be repeated in each local time step $[t_n, t_{n+1}]$.
{
How the scheme is actually implemented will be illustrated in details in the next section.}
In summary, the strategy of InDC methods is to use a simple numerical method to compute numerical solutions $\hat{y}^{(0)}(t)$ and $\hat{z}^{(0)}(t)$ as prediction, and then to solve a series of correction equations in the integral form based on equations (\ref{defint}). Each correction improves the accuracy of numerical solutions from the previous iteration.    

In our description of InDC, we let $\hat{y}^{(k)}_m$, $\hat{z}^{(k)}_m$, $\hat{e}^{(k)}_m$, $\hat{d}^{(k)}_m$ denote the numerical approximations (with hat) to the exact solutions and error functions. We use subscript $m$ to denote the location $\tau_m$ of the quadrature points and use superscript $(k)$ to denote the prediction ($k=0$) and correction loops ($k=1, \cdots$). We let $\underline{\cdot}$ denote the vector on InDC quadrature nodes $(\tau_1, \cdots \tau_M)$, for example, $\underline{y} = (y_1, \cdots, y_M)$.

\subsection{InDC-IMEX1 method}

{In this subsection, we consider InDC method constructed using implicit-explicit Euler method, applied to (\ref{spp}).}

We use uniformly distributed quadrature nodes $\tau_1,...,\tau_M$ from equation \eqref{eq: node} excluding the left-most endpoint.
\begin{enumerate}
\item 
(Prediction step) Use the first implicit-explicit Euler  method to compute 

$$\underline{\hat{y}}^{(0)} = (\hat{y}^{(0)}_1,...,\hat{y}^{(0)}_m,...,\hat{y}^{(0)}_M), \quad
\underline{\hat{z}}^{(0)} = (\hat{z}^{(0)}_1,...,\hat{z}^{(0)}_m,...,\hat{z}^{(0)}_M)
$$ 
as the approximations of the exact solution {$\underline{y} = (y_1, \cdots,y_m, \cdots,y_M)$ and $\underline{z} = (z_1, \cdots, z_m, \cdots, z_M)$} at quadrature nodes $\tau_1,...,\tau_M$. 
This gives
\beq
\label{eq: Euler-eps}
\begin{array}{l} 
\hat{y}^{(0)}_{m+1} = \hat{y}^{(0)}_m + h f(\hat{y}^{(0)}_{m},\hat{z}^{(0)}_{m}),\\[2mm]
\varepsilon \hat{z}^{(0)}_{m+1} = \varepsilon \hat{z}^{(0)}_m + h g(\hat{y}^{(0)}_{m+1},\hat{z}^{(0)}_{m+1}),
\end{array}
\eeq
for $m = 0,1,...M-1$.
\item
(Correction loop). For $k = 1,...,K$ ($K$ is the number of the correction step), 
let $\hat{y}^{(k-1)}$ and $\hat{z}^{(k-1)}$ denote the numerical solutions at the $(k-1)^{th}$ correction.
\begin{enumerate}
\item Denote the error function at the $(k-1)^{th}$ correction $e^{(k-1)}(t) = y(t) - \hat{y}^{(k-1)}(t)$, where $y(t)$
is the exact solution and $\hat{y}^{(k-1)}(t)$ is a $(M-1)^{th}$ order polynomial interpolating  $\bar{\hat{y}}^{(k-1)}$ at quadrature nodes $\tau_1,...,\tau_M$. Similarly denote $d^{(k-1)}(t) = z(t) - \hat{z}^{(k-1)}(t)$. {The initial conditions for these error functions are zero, i.e. $\hat{e}^{(k-1)}_0 = 0$ and $\hat{d}^{(k-1)}_0 = 0, \forall k$.}
Let $\delta^{(k-1)}(t)$ and $\rho^{(k-1)}(t)$, $\Delta f^{(k-1)}$ and $\Delta g^{(k-1)}$ be defined by equations \eqref{eq: delta_diff} and \eqref{eq: def_deltafg} respectively, but with the upper script $(0)$ replaced with $(k-1)$.
We compute the numerical error vector $\bar{\hat{e}}^{(k-1)} = (\hat{e}_1^{(k-1)},...,\hat{e}_M^{(k-1)})$ with $\hat{e}_m^{(k-1)}$ approximating $e^{(k-1)}(\tau_m)$ by applying the first order IMEX R-K method with Butcher table specified in \eqref{first_schemeARS} to the integral form of (\ref{defint}),
\beq
\label{errorEqs}
\begin{array}{lll}
\displaystyle \hat{e}^{(k-1)}_{m+1} &=& \hat{e}^{(k-1)}_m + h \Delta f^{(k-1)}_{m} 
+ \int_{\tau_m}^{\tau_{m+1}}\delta^{(k-1)}(s)ds,\\[2mm]
\displaystyle \varepsilon \hat{d}^{(k-1)}_{m+1} & =& \varepsilon \hat{d}^{(k-1)}_m +  h \Delta g^{(k-1)}_{m+1}
+\int_{\tau_m}^{\tau_{m+1}}\rho^{(k-1)}(s)ds,
\end{array}
\eeq
where 
\beq
\label{appint}
\begin{array}{l}
\int_{\tau_m}^{\tau_{m+1}}\delta^{(k-1)}(s)ds  =  \int_{\tau_m}^{\tau_{m+1}}f(\hat{y}^{(k-1)}(s),\hat{z}^{(k-1)}(s))ds - \hat{y}^{(k-1)}_{m+1} + \hat{y}^{(k-1)}_{m},\\[2mm]
\int_{\tau_m}^{\tau_{m+1}}\rho^{(k-1)}(s)ds = \int_{\tau_m}^{\tau_{m+1}}g(\hat{y}^{(k-1)}(s),\hat{z}^{(k-1)}(s))ds - \varepsilon \hat{z}^{(k-1)}_{m+1} + \varepsilon \hat{z}^{(k-1)}_{m}.
\end{array}
\eeq
Integral terms $\int_{\tau_m}^{\tau_{m+1}}$ on the right hand side of equations \eqref{appint} are approximated by a numerical quadrature, {in particular by an interpolary quadrature formula}.  
Especially, let $S$ be the integration matrix, its $(m, l)$ element is
\[
S^{m,l} = \frac{1}{h}\int_{\tau_m}^{\tau_{m+1}}\alpha_l(s) d s, \quad \mbox{for} \quad m=0, \cdots, M-1, \quad l=1, \cdots M,
\] 
where $\alpha_l(s)$ is the Lagrangian basis function based on the node $\tau_l$, $l=1, \cdots M$. 
{Analytical expression of $S^{m,l}$ can be given, which does not depends on $h$, 
 \[
S^{m,l} = \int_{m}^{{m+1}}\prod_{j \neq l}\frac{\theta-j}{l-j} d \theta, \quad \mbox{for} \quad m=0, \cdots, M-1, \quad l=1, \cdots M,
\] 
 }
Let 
{\beq
\label{eq: Sm}
S^m(\underline{f}) = \sum_{j = 1}^{M} S^{m,j}f(y_j,z_j),
\eeq
then 
\[
hS^m(\underline{f})-\int_{\tau_m}^{\tau_{m+1}}f(y(s), z(s))ds = \mathcal{O}(h^{M+1}),
\] }
for any smooth function $f$.  In other words, the quadrature formula given by $hS^{m}(\underline{f})$ approximates the exact integration with $(M+1)^{th}$ order of accuracy {\em locally}.
\item Update the approximate solutions 
\beq
\label{eq: update}
\underline{\hat{y}}^{(k)} = \underline{\hat{y}}^{(k-1)}+ \underline{\hat{e}}^{(k-1)}, \quad \underline{\hat{z}}^{(k)} = \underline{\hat{z}}^{(k-1)}+  \underline{\hat{d}}^{(k-1)}.
\eeq
\end{enumerate}
{Note that from equations (\ref{errorEqs}), \eqref{appint}, \eqref{eq: update} and using the notation introduced in equation \eqref{eq: Sm},} we get
\beq
\label{errorEqs2}
\begin{array}{l} 
\hat{y}^{(k)}_{m+1} = \hat{y}^{(k)}_m + h \Delta f_{m}^{(k-1)} + h S^{m}(\underline{\hat{f}}^{(k-1)}),\\[2mm]
\varepsilon \hat{z}^{(k)}_{m+1} =\varepsilon  \hat{z}^{(k)}_m + h \Delta g_{m+1}^{(k-1)}  + h S^{m}(\underline{\hat{g}}^{(k-1)}).
\end{array}
\eeq
\end{enumerate}
This completes the description of the {InDC-IMEX1} method. 

Note that in practice $k$ is chosen in such a way that there is no further improvement in order of convergence when increasing it. The precise relation between $k$ and $M$ and the order of accuracy of the method will be reported later.

\begin{rem}\label{RGSA}
{The assumption of 
 GSA guarantees, in the case $\varepsilon = 0$ (as in (\ref{errorEqs2})}), that the approximate solutions in the prediction step $k = 0$ satisfy the equation $g(\hat{y}^{(0)}_m, \hat{z}^{(0)}_m) = 0$,  and this implies $\hat{z}^{(0)}_m = \mathcal{G}(\hat{y}^{(0)}_m) $, $\forall m$ {(see Remark \ref{rem: index1})}.

Consequently in the first correction step for $k = 1$, by (\ref{eq: def_deltafg}), the second equation in system (\ref{errorEqs2}) is reduced to 
\begin{eqnarray}\label{principalDAE2}
\begin{array}{c}
\Delta g^{(0)}_{m} = 0 \quad \ \forall m = 0, \cdots, M.
\end{array}
\end{eqnarray}
and this guarantees that {$\underline{\hat{g}}^{(1)} =0$, i.e. $g(\hat{y}^{(1)}_m, \hat{z}^{(1)}_m) =  0$, $\forall m$}. By Remark \ref{rem: index1}, we get $\hat{z}^{(1)}_m = \mathcal{G}(\hat{y}^{(1)}_m), \forall m = 0, \cdots, M$. A similar argument can be given to show that $\underline{\hat{g}}^{(k)} = 0$, for $k=2, \cdots ,K$.
Hence, again by Remark \ref{rem: index1}, we have $g(\hat{y}^{(k)}_{m,0}, \hat{z}^{(k)}_{m,0}) = 0$, i.e., 
\[\hat{z}_{m,0}^{(k)} = \mathcal{G}(\hat{y}_{m,0}^{(k)}), \ \forall k=0, \cdots K, \ \forall m = 0, \cdots,  M.
\]
Therefore, the updating of $\hat{y}_{m+1}^{(k)}$ represents the correction step for a non-stiff ordinary differential equation of the form (\ref{eqy}), i.e., 
$$
\hat{y}^{(k)}_{m+1} = \hat{y}^{(k)}_m + h ({f}(\hat{y}_{m}^{(k)}) - \hat{f}(y_{m}^{(k-1)}) )  + h S^{m}(\bar{\hat{f}}^{(k-1)})
$$
where $f(\cdot) = f(\cdot, \mathcal{G}(\cdot))$.
Then by Remark \ref{rem: index1} b) and by classical results in the InDC framework for non-stiff ordinary differentia equations (Theorem 4.1 in \cite{christlieb2009integral} for details), we obtain the classical increasing order of accuracy for the $y$ and $z$-component, i.e., 
\begin{eqnarray}
\label{eq: lemma2_IMEX-E}
\begin{array}{l}
e^{(K)}_n \doteq {\hat{y}^{(K)}_n- y(t_n)} = \mathcal{O}(H^{\min(k+2, M)}), 
\quad
d^{(k)}_n \doteq {\hat{z}^{(K)}_n- z(t_n)} =\mathcal{O}(H^{\min(k+2, M)}).
\end{array}
\end{eqnarray}
with $k\le M$.
In a straightforward manner, in the next section we will generalize this result.
\end{rem}

{In the next section we show that InDC IMEX R-K methods can be rewritten as standard IMEX R-K schemes.}

\subsection{Reformulation of InDC-IMEX methods}
\label{InDC_SA}
In \cite{christlieb2009comments}, the InDC method constructed with explicit or implicit R-K methods in the prediction and correction steps has been reformulated as a high-order explicit or implicit R-K method whose Butcher tableau is explicitly constructed. 

{Similarly in this section we show as an InDC method constructed by a standard IMEX R-K method can be re-written as an IMEX R-K one. We first present the reformulation of InDC-IMEX1 methods constructed from first order classical IMEX R-K schemes as an IMEX method with the corresponding Butcher tableau for the explicit and implicit parts.} 

{We start to give a list of a classical first order IMEX R-K schemes given by a double Butcher tableau  (\ref{DBT})  corresponding to the implicit and explicit part of the method for different types, (\ref{Def:A}) and (\ref{Def:CK}). Furthermore we emphisize which of these schemes are GSA. 
\begin{itemize}
\item First order IMEX R-K method of type ARS (particular case of CK), it is GSA, \cite{ascher1997implicit}:
\begin{eqnarray}
\label{first_schemeARS}
\begin{array}{c|cc}
              0 & 0 &0 \\
              1 & 1 &0\\
              \hline
               &1 & 0 
\end{array} \qquad
\begin{array}{c|cc}
               0 &  0 & 0\\
               1 & 0 & 1\\
              \hline 
               & 0 & 1
\end{array}.
\end{eqnarray}
\item First order IMEX R-K method of type A, it is GSA:
\begin{eqnarray}
\label{first_schemeA_GSA}
\begin{array}{c|cc}
              0 & 0 &0 \\
              1 & 1 &0\\
              \hline
               &1 & 0 
\end{array} \qquad
\begin{array}{c|cc}
               1 &  1 & 0\\
               1 & 0 & 1\\
              \hline 
               & 0 & 1
\end{array}.
\end{eqnarray}
\item First order IMEX R-K method of type A, it is not GSA, \cite{pareschi2005implicit}:
\begin{eqnarray}\label{first_schemeA}
\begin{array}{c|c}
              0 & 0 \\
              \hline
               & 1
\end{array} \qquad
\begin{array}{c|c}
               1 & 1 \\
              \hline 
               & 1 
\end{array}.
\end{eqnarray}
\end{itemize}
}
Note that scheme (\ref{first_schemeARS}) is the implicit-explicit Euler method, i.e.\ Eq.~(\ref{eq: Euler-eps}), applied to problem (\ref{spp}). Scheme (\ref{first_schemeA_GSA}) is not used in the literature, because it not efficient. It is reported here only as an example of first order type A scheme which is GSA. As we  shall see, InDC IMEX methods based on GSA schemes maintain the accuracy  as $\varepsilon \to 0$, while the InDC IMEX R-K method based on a non GSA scheme (\ref{first_schemeA}) shows a degradation of accuracy.

We first present the reformulation of InDC-IMEX1 method constructed by scheme (\ref{first_schemeARS}), and prove that the corresponding InDC is an IMEX R-K method of type ARS and it is GSA. We call it  {InDC-IMEX1-GSA-ARS}. 

First we present the Butcher tableau for the {InDC-IMEX1-GSA-ARS} method with one loop of correction step as an example. 
This takes the form
\begin{align}\label{CBTab}
  \begin{array}{c|ccc}
  0 & 0 & {\bf 0}^T & {\bf 0}^T\\
       \vec{{c}} &\frac{\bf 1}M & \tilde{T} & O\\
    \vec{{c}} & {\bf 0} & \tilde{P} & \tilde{T} \\
    \hline
    & 0& {\tilde{\bf b}}_1^T & {\tilde{\bf b}}_2^T 
  \end{array}
  \quad
  \begin{array}{c|ccc}
    0 & 0 &{\bf 0}^T & {\bf 0}^T \\
    \vec{c}& {\bf 0} & T & O\\
    \vec{c} & {\bf 0} &P & T \\
    \hline
    & 0 & \vec{b}_1^T & \vec{b}_2^T 
  \end{array}  
\end{align}
where 
$
  \vec{c} = \frac1M \left[1, \cdots, M \right]^T,
$
${\bf 0}$ and ${\bf 1}$ are vectors with $0$ and $1$ entries respectively, having the same size as ${\bf c}$. 
${O}$ is a $M\times M$ matrix of zeros,  
$T$, $\tilde{T}$, $P$ and $\tilde{P}$ are $M \times M$ matrices,
with
 \begin{align*}
   \tilde{T}= \frac1M \left[
  \begin{array}{ccccc}
     0 & 0 & 0&\ldots&0\\ 
     1 & 0 & 0& \ldots&0\\
    \vdots & \vdots & \ddots & \vdots & \vdots \\
     1 & 1 & 1 &\ldots&0
  \end{array}
  \right],
  \quad
     {T}= \frac1M \left[
  \begin{array}{ccccc}
     1 & 0 & 0&\ldots&0\\ 
     1 & 1 & 0& \ldots&0\\
    \vdots & \vdots & \ddots & \vdots & \vdots \\
     1 & 1 & 1 &\ldots&1
  \end{array}
  \right],
\end{align*}
and $\tilde{P} = S - \tilde{T}, \quad P = S - T$,
where the matrix $S$ is constructed such that its entries  
{${S}_{ij}=\int_{t_0}^{t_i} \alpha_j(s)ds = \sum_{k = 0}^{i-1} S^{k,j}$} with $\alpha_j(s)$ the Lagrangian basis functions for the node $\tau_j$. The vectors $\tilde{\bf b}_1^T$, $\tilde{\bf b}_2^T$, ${\bf b}_1^T$ and ${\bf b}_2^T$ are taken so that they are the last rows of the matrices $\tilde{P}$, $\tilde{T}$, $P$ and $T$.
Furthermore, the method (\ref{CBTab}) has the same nodes $\vec{{c}}$ in implicit and explicit parts and the weights $\tilde{\bf b}_i^T \neq  {\bf b}_i^T$, for $i = 1,2$. 
As an example, we show the Butcher table for InDC-IMEX1-GSA-ARS method with $M=2$ and with one correction loop below. {Note that $S = [3/4, -1/4; 1, 0]$ for $M=2$.} 
\begin{eqnarray}
\label{eq: semi-implicit-InDC_A}
\begin{array}{c|ccccc}
              0    & 0    &0    & 0 & 0 &0 \\
              1/2 & 1/2 &0    & 0 & 0 &0\\
              1    & 1/2 &1/2 &0  & 0 &0 \\
             1/2  & 0    & 3/4&-1/4&0&0\\
              1 &0& 1/2 & 0& 1/2&0\\
              \hline
              & 0 & 1/2 & 0& 1/2&0
\end{array} \quad
\begin{array}{c|ccccc}
              0    & 0    &0    & 0 & 0 &0 \\
              1/2 & 0 &1/2    & 0 & 0 &0\\
              1    & 0 &1/2 &1/2  & 0 &0 \\
             1/2  & 0    & 1/4&-1/4&1/2&0\\
              1 &0& 1/2 & -1/2& 1/2&1/2\\
              \hline
              & 0 & 1/2 & -1/2& 1/2&1/2
\end{array}.
\end{eqnarray}


The same conclusion holds for the InDC method constructed with the first order IMEX R-K schemes (\ref{first_schemeA_GSA}). We obtain a InDC IMEX R-K scheme denoted as InDC-IMEX1-GSA-A. It can be shown, by similar argument presented for the previous scheme, that the assembled matrix $\mathbb{A}$ of the implicit part in the table Butcher tableau, for the InDC-IMEX1-GSA-A method, is invertible, and, at the end, we obtain a IMEX R-K method of type A which is GSA. 

As an example, we show below the double Butcher tableaus for the InDC-IMEX1-GSA-A method with $M=2$ and with one loop of correction only. The tableaus, obtained by an using an algebraic manipulator from  Eq.~(\ref{errorEqs2}), show that the InDC IMEX R-K scheme is of type A and is GSA.
{
\begin{eqnarray}
\label{eq: semi-implicit-InDC_AA}
\begin{array}{c|cccccccc}
              0 & 0 &0 & 0 & 0& 0 &0 & 0 & 0\\
              1/2 & 1/2 &0 & 0 & 0& 0 &0 & 0 & 0\\
              1/2 & 1/2 &0 & 0 & 0& 0 &0 & 0 & 0\\
              1 & 1/2 &0 & 1/2 & 0& 0 &0 & 0 & 0\\
              0 & 0 &0 & 0 & 0& 0 &0 & 0 & 0\\
             1/2&0&3/4&0&-1/4&0&0&0&0\\
             1/2&0&3/4&0&-1/4&0&0&0&0\\
              1 &0&  1/2 & 0& 0&0&1/2&0&0\\
              \hline
              1 &0&  1/2 & 0& 0&0&1/2&0&0\\
\end{array} \quad
\begin{array}{c|cccccccc}
              1/2 & 1/2 &0 & 0 & 0& 0 &0 & 0 & 0\\
              1/2 & 0&1/2 & 0 & 0& 0 &0 & 0 & 0\\
              1 & 0& 1/2 & 1/2 & 0& 0 &0 & 0 & 0\\
              1 & 0& 1/2 &0 & 1/2 & 0 &0 & 0 & 0\\
             1/2&0 & 1/4&0& -1/4&1/2&0&0&0\\
             1/2&0& 1/4&0&-1/4&0&1/2&0&0\\
              1 &0& 1/2 & 0& -1/2 &0&1/2&1/2&0\\
              1 &0& 1/2 & 0& -1/2 &0&1/2&0&1/2\\
              \hline
               &0& 1/2 & 0& -1/2 &0&1/2&0&1/2\\
\end{array}.
\end{eqnarray}
}

Finally, we consider an InDC IMEX R-K method constructed by using the first order IMEX scheme (\ref{first_schemeA}) which is not GSA. The InDC method is denoted by InDC-IMEX1-NGSA-A. For such an scheme, we show that the assembled matrix $\mathbb{A}$ of the implicit part in the Butcher table is non-invertible. 

Similaly as an example, we consider InDC-IMEX1-NGSA-A constructed by using $M=2$ and with one loop of correction only.}

Then, the corresponding Butcher tables can be assembled as
\begin{eqnarray}
\label{eq: semi-implicit-InDC_B}
\begin{array}{c|cccccccc}
              0 & 0 &0 & 0 & 0& 0 &0 & 0 & 0\\
              1/2 & 1/2 &0 & 0 & 0& 0 &0 & 0 & 0\\
              1/2 & 1/2 &0 & 0 & 0& 0 &0 & 0 & 0\\
              1 & 1/2 &0 & 1/2 & 0& 0 &0 & 0 & 0\\
              0 & 0 &0 & 0 & 0& 0 &0 & 0 & 0\\
             1/2&-1/2&3/4&0&-1/4&1/2&0&0&0\\
             1/2&-1/2&3/4&0&-1/4&1/2&0&0&0\\
              1 & -1/2 & 1& -1/2&0&1/2&0&1/2&0\\
              \hline
              & -1/2 & 1& -1/2&0&1/2&0&1/2&0\\
\end{array} \quad
\begin{array}{c|cccccccc}
              1/2 & 1/2 &0 & 0 & 0& 0 &0 & 0 & 0\\
              1/2 & 1/2 &0 & 0 & 0& 0 &0 & 0 & 0\\
              1 & 1/2 &0 & 1/2 & 0& 0 &0 & 0 & 0\\
              1 & 1/2 &0 & 1/2 & 0& 0 &0 & 0 & 0\\
             1/2&-1/2&3/4&0&-1/4&1/2&0&0&0\\
             1/2&-1/2&3/4&0&-1/4&1/2&0&0&0\\
              1 & -1/2 & 1& -1/2&0&1/2&0&1/2&0\\
              1 & -1/2 & 1& -1/2&0&1/2&0&1/2&0\\
              \hline
              & -1/2 & 1& -1/2&0&1/2&0&1/2&0\\
\end{array}.
\end{eqnarray}

Notice that the size of the Butcher matrices (\ref{eq: semi-implicit-InDC_AA}) and (\ref{eq: semi-implicit-InDC_B}) is $2*M*(K+1)$, in comparison with $M*(K+1)$ for the $\hat{\mathbb{A}}$ matrix for InDC-IMEX1-GSA-ARS methods, where the factor $2$ is due to the fact that we need an extra row to compute solutions at InDC quadrature points. When the method is not GSA, the last stage of IMEX is not the updated solution at quadrature nodes, hence we need one more row to represent solutions at quadrature nodes, see Row 2, 4, 6 and 8 in the implicit matrix presented in \eqref{eq: semi-implicit-InDC_B}. Besides these extra rows, the construction of the assembled Butcher tableau is in a similar way as those in eq.~\eqref{eq: semi-implicit-InDC_A} and \eqref{eq: semi-implicit-InDC_AA}. Notice that in matrix (\ref{eq: semi-implicit-InDC_B}) for row 2, 4, 6, 8, the diagonal entries are $0$ for the implicit part of the Butcher Tableau, i.e. the assembled matrix of the implicit part is not invertible.


Then for the InDC-IMEX1-GSA schemes constructed with first order IMEX schemes GSA (\ref{first_schemeARS})-(\ref{first_schemeA_GSA}), we have the following proposition regarding the invertibility of the assembled matrix of the implicit part.

\begin{prop} 
\label{prop: inver} (Invertibility of implicit assembled matrix)
{Consider the InDC method constructed with a first order IMEX R-K methods of type A or CK.
If the quadrature nodes used in the InDC method exclude the left-most point and the first order IMEX R-K method is GSA, then the InDC-IMEX1-GSA method is an IMEX R-K of type A or CK and it is GSA with the assembled matrix $\mathbb{A}$ or $\hat{\mathbb{A}}$ in the implicit part being invertible.
}
\end{prop}
\noindent
{\em Proof.} First the statement is true for the scheme with only one correction loop, as seen from the Butcher tables presented in eq.~\eqref{eq: semi-implicit-InDC_A}. Such results can be generalized to the InDC method with general $K$ correction loops, with a lower triangular matrix {$\hat{\mathbb{A}}$} of size $M(K+1)$ and 
\begin{align}\label{hatA2}
{\hat{\mathbb{A} }} = 
\left (
\begin{array}{ccccc}
     T & &&&\\
     P & T & &&\\
     { O}&P&T&&\\
     &&\cdots&&\\
     { O} & \cdots { O}& P& T
  \end{array}  
  \right),
\end{align}
whose diagonal entries are diagonal entries of $T$ and are $\frac1M$, i.e. nonzero. Hence we have the invertibility of $\hat{\mathbb{A}}$. 
Notice that the InDC method is also GSA, i.e. the last rows of the assembled matrices $\tilde{\mathbb{A}}$ and $\hat{\mathbb{A}}$ are the same as the $\tilde{\bf b}^T$ and ${\bf b}^T$ vectors.

In a similar fashion, we prove for IMEX schemes of type A, where the assembled matrix $\mathbb{A}$ is similar to (\ref{hatA2}), with size $2M(K+1).$ $\Box$

In the next section we extend this previous result for InDC IMEX R-K schemes constructed by considering arbitrary order $s$-stage IMEX R-K.

\subsection{InDC Methods constructed with IMEX R-K methods}

We describe in general how we apply $s$-stage IMEX R-K methods in correction loops in an InDC framework. 
We assume double Butcher tables for the IMEX R-K method as specified in Section \ref{sec:2}. 
For the internal stages, we introduce the integration and interpolation matrices, {which are based on the use of the same set of nodes for each interval, more precisely, in interval $[\tau_m, \tau_{m+1}]$ we assume that the nodes are located at $\tau_m + c_i h$, where $c_{i}$, $i = 1,\cdots s$, do not depend on the interval $m$}. Then we have:
{
\beq
\label{eq: s_cmi}
h S^{c_{i}, k} =  \int_{\tau_m}^{\tau_m+c_{i} h}\alpha_k(s) d s, 
\quad
P^{c_{i}, k} = \alpha_k(\tau_m+c_{i} h),
\eeq
}
{Similarly:
\beq
\label{eq: s_cmi_ex}
h S^{\tilde{c}_{i}, k} =  \int_{\tau_m}^{\tau_m+\tilde{c}_{i} h}\alpha_k(s) d s, 
\quad
P^{\tilde{c}_{i}, k} = \alpha_k(\tau_m+\tilde{c}_{i} h),
\eeq
}
$\forall m=0, \cdots, M-1, \quad \forall k=1, \cdots M, \quad \forall i = 1, \cdots s$,
where $\alpha_j(s)$ is the Lagrangian basis function based on the node $\tau_j$.
Let 
\[
S^{c_{i}}(\underline{f}) = \sum_{j = 1}^{M} S^{c_{i}, j} f(y_j,z_j), \quad P^{c_{i}}(\underline{f}) = \sum_{j = 1}^{M} P^{c_{i}, j} f(y_j,z_j),
\]
\[
S^{\tilde{c}_{i}}(\underline{f}) = \sum_{j = 1}^{M} S^{\tilde{c}_{i}, j} f(y_j,z_j), \quad P^{\tilde{c}_{i}}(\underline{f}) = \sum_{j = 1}^{M} P^{\tilde{c}_{i}, j} f(y_j,z_j).
\]
Then 
\[
hS^{c_{i}}(\underline{f})-\int_{\tau_m}^{\tau_{m}+c_i h}f(y(s), z(s))ds = \mathcal{O}(h^{M+1}),
\quad
P^{c_{i}}(\underline{f})-f(y(\tau_{m}+c_i h), z(t_{m}+c_i h)) = \mathcal{O}(h^{M}),
\]
and similarly,
\[
hS^{\tilde{c}_{mi}}(\underline{f})-\int_{\tau_m}^{\tau_{m}+\tilde{c}_i h}f(y(s), z(s))ds = \mathcal{O}(h^{M+1}),
\quad
P^{\tilde{c}_{mi}}(\underline{f})-f(y(\tau_{m}+\tilde{c}_i h), z(\tau_{m}+\tilde{c}_i h)) = \mathcal{O}(h^{M}),
\]
for any smooth function $f$. In other words, the quadrature formula given by $hS^{c_{i}}(\underline{f})$ approximates the exact integration with $(M+1)^{th}$ order accuracy locally, while the interpolation formula given by  $P^{c_{i}}(\underline{f})$ approximates the exact solution at R-K internal stages with $M^{th}$ order accuracy.

We compute the numerical errors approximating the error functions $e^{(k-1)}(\tau_m)$, $d^{(k-1)}(\tau_m)$ with an IMEX R-K method to (\ref{defint}), 
\beq
\label{princNum}
 \left(\begin{array}{c}
\displaystyle\hat{e}_{m+1}^{(k-1)}\\
\displaystyle\eps \hat{d}_{m+1}^{(k-1)}
\end{array}\right)
 = \left(\begin{array}{c}
\displaystyle \hat{e}_{m}^{(k-1)} + h\int_{0}^{1}\delta(\tau_m + \tau h)d\tau\\
\displaystyle \eps\hat{d}_{m}^{(k-1)} + h \int_{0}^{1}\rho(\tau_m + \tau h)d\tau
\end{array}\right) 
 + h  \left(\begin{array}{c}
                    \displaystyle    \sum_{i=1}^{s} \tilde{b}_i \ \Delta \hat{\mathcal{K}}^{(k-1)}_{mi}\\
                     \displaystyle    \sum_{i=1}^{s} b_i  \ \Delta \hat{\mathcal{L}}^{(k-1)}_{mi}
                          \end{array}\right), 
\eeq
with
\beqa 
\label{EvF}
\left(\begin{array}{c}
\Delta \hat{\mathcal{K}}^{({k-1})}_{mi}\\
\Delta \hat{\mathcal{L}}^{(k-1)}_{mi}
\end{array}\right)
& \doteq& \left(\begin{array}{c}
f(\hat{Y}^{(k)}_{mi} , \hat{Z}^{(k)}_{mi})-{P^{\tilde{c}_{i}}(\underline{\hat{f}}^{(k-1)})} \\
g(\hat{Y}^{(k)}_{mi}, \hat{Z}^{(k)}_{mi})-P^{c_{i}}(\underline{\hat{g}}^{(k-1)})
\end{array}\right).
\label{EvF2}
\eeqa
Here we set
\beq
\label{Newk}
\hat{Y}^{(k)}_{mi} = { P^{\tilde{c}_{i}} (\underline{\hat{y}}^{(k-1)})}+ \hat{E}^{(k-1)}_{mi}, \quad \hat{Z}^{(k)}_{mi} = P^{c_{i}} (\underline{\hat{z}}^{(k-1)}) + \hat{D}^{(k-1)}_{mi},
\eeq
with
\beq\label{inter_stag}
\left(\begin{array}{c}
\displaystyle\hat{E}_{mi}^{(k-1)}\\
\displaystyle\eps\hat{D}_{mi}^{(k-1)}
\end{array}\right)
 = \left(\begin{array}{c}
\displaystyle\hat{e}_{m}^{(k-1)} + h \int_{0}^{{\tilde{c}_{i}}}{\delta}(\tau_m + \tau h)d\tau\\
\displaystyle\eps\hat{d}_{m}^{(k-1)} + h \int_{0}^{c_{i}}{\rho}(\tau_m + \tau h)d\tau\\
\end{array}\right)
 + h \left(\begin{array}{c}
\displaystyle             \sum_{j=1}^{i}          \tilde{a}_{ij} \  \Delta \hat{\mathcal{K}}^{({k-1})}_{mj}\\
\displaystyle               \sum_{j=1}^{i}          a_{ij} \ \Delta \hat{\mathcal{L}}^{({k-1})}_{mj}
                          \end{array}\right).
\eeq
The integral terms in the system \eqref{princNum} can be approximated by quadrature rules, then we can rewrite (\ref{princNum}) and (\ref{inter_stag}) as
\beq\label{newK}
\left(\begin{array}{c}
\displaystyle\hat{y}_{m+1}^{(k)} -  h S^{m,(k-1)}_{\underline{\hat{{f}}}}\\
\displaystyle\eps\hat{z}_{m+1}^{(k)} - h S^{m,(k-1)}_{\underline{\hat{{g}}}}
\end{array}\right)
 = \left(\begin{array}{c}
\displaystyle\hat{y}_{m}^{(k)}\\
\displaystyle\eps\hat{z}_{m}^{(k)} 
\end{array}\right) 
 + h  \left(\begin{array}{c}
           \displaystyle           \sum_{i=1}^{s}  {\tilde{b}_i}  \ \Delta \hat{\mathcal{K}}^{({k-1})}_{mi}\\
            \displaystyle          \sum_{i=1}^{s}   b_i \   \Delta \hat{\mathcal{L}}^{({k-1})}_{mi}
                          \end{array}\right), 
\eeq
\beq\label{newapproach}
\left(\begin{array}{c}
\displaystyle\hat{Y}_{mi}^{(k)}- {hS^{\tilde{c}_{i},(k-1)}_{\underline{\hat{f}}}}\\
\displaystyle\eps\hat{Z}_{mi}^{(k)}-  hS^{c_{i},(k-1)}_{\underline{\hat{g}}}
\end{array}\right)
 = \left(\begin{array}{c}
\displaystyle\hat{y}_{m}^{(k)} \\
\displaystyle\eps \hat{z}_{m}^{(k)} 
\end{array}\right)
 + h   \left(\begin{array}{c}
            \displaystyle     \sum_{j=1}^{i}       \tilde{a}_{ij}   \Delta \hat{\mathcal{K}}^{({k-1})}_{mj}\\
            \displaystyle      \sum_{j=1}^{i}       a_{ij} \Delta \hat{\mathcal{L}}^{({k-1})}_{mj}
                          \end{array}\right),
\eeq
with
\beq \label{newS}
\left(\begin{array}{c}
  S^{m,(k-1)}_{\underline{\hat{f}}} = S^m(\underline{\hat{f}}^{(k-1)})\\
 \eps S^{m,(k-1)}_{\underline{\hat{g}}} =  S^m(\underline{\hat{g}}^{(k-1)})
                            \end{array}\right),\quad
\left(\begin{array}{c}
{S^{\tilde{c}_{i},(k-1)}_{\underline{\hat{f}}} = S^{\tilde{c}_{i}}(\underline{\hat{f}}^{(k-1)})}\\
 \eps S^{c_{i},(k-1)}_{\underline{\hat{g}}} =  S^{c_{i}}(\underline{\hat{g}}^{(k-1)})
                            \end{array}\right).                           
\eeq 

{The double Butcher tableaus for the InDC methods constructed with high order IMEX R-K can be assembled just as done for the InDC-IMEX1 method. They are of size related to the number of quadrature points $M$ as well as the number of stages $s$ of the IMEX R-K method. {Below we present Proposition \ref{prop: inver_ARK} as a general result regarding the invertibility of the implicit part of assembled Butcher tableau. We choose not to present the assembled Butcher table for the InDC method constructed with high order IMEX R-K to save space. For example, to construct a Butcher table of InDC IMEX2 method with four quadrature points and one correction loop, the assembled matrix will be at least of size $17\times 17$. For the construction of classical InDC methods using high order R-K methods, the reader can consult the reference \cite{christlieb2009comments}.}

\begin{prop} 
\label{prop: inver_ARK}
(Invertibility of the implicit assembled matrices ${\mathbb{A}}$ $\hat{\mathbb{A}}$) Consider the InDC method constructed with IMEX R-K methods of type A or CK with invertible matrix $A$ or $\hat{A}$ in implicit part of Butcher tableau, respectively. If the quadrature nodes used in the InDC method exclude the left-most point and the IMEX R-K method is GSA, then the InDC method is a IMEX R-K method of type A or CK and it is GSA with the assembled matrix ${\mathbb{A}}$ or $\hat{\mathbb{A}}$ in the implicit part being invertible. 
\end{prop}
\noindent
\begin{proof} {The proof is similar to that for Proposition~\ref{prop: inver}. We first consider the case of type ARS, which is a special case of type CK. We let $\hat{\mathbb{A}}$ with size $\sigma$ be the invertible sub-matrix in the implicit part of the InDC Butcher table and $\hat{A}$ with size $s$ be the invertible matrix in the implicit part of IMEX method used to construct the InDC-IMEX method. Then $\sigma = s*M*(K+1)$, where $M$ is the number of quadrature points and $K$ is the number of correction loops. For example, for the InDC method constructed with IMEX2 method \eqref{SesondARS} with $M=2$ and $K=1$, $s=2$, then $\sigma = 8$. $\hat{\mathbb{A}}$ can be constructed with a similar structure as shown in eq.~\eqref{hatA2}, where $T$ is of size $s*M$ and it's block triangular. 
Its diagonal blocks ($M$ of them) are matricies $\hat{A}$ {of the IMEX R-K scheme} (of size $s$) scaled by the size of subinterval $1/M$. Recall from classical linear algebra that, the determinant of a block triangular matrix is the product of the determinants of diagonal blocks. Therefore, 
\[
det(T) = (\frac1M det({\hat{A}}))^M \neq0,
\]
due to the invertibility of the matrix $A$ of the IMEX R-K method used to construct the InDC IMEX one. Hence, $det(\hat{\mathbb{A}}) = det(T)^{K+1} \neq0$, i.e. $\hat{\mathbb{A}}$ is invertible.
%
Finally, $\tilde{\bf b}^T$ and ${\bf b}^T$ vectors come from the last row of Butcher tableaus $\tilde{\mathbb{A}}$ and $\hat{\mathbb{A}}$, therefore the InDC method is GSA too. 
Similar results can be obtained for the type A and type CK.
}
\end{proof}
\begin{rem} \label{InDC_SA2}
Note that the assumptions in Proposition~\ref{prop: inver} and \ref{prop: inver_ARK} provide a sufficient condition to guarantee the invertibility of implicit assembled matrix $\mathbb{A}$ or $\hat{\mathbb{A}}$. In particular, it has been pointed out in \cite{boscarino2008error} that the {invertibility of matrix $A$ or  submatrix $\hat{A}$ of the IMEX R-K method}, is an essential hypothesis for the error analysis for IMEX R-K methods {of different types}, otherwise error estimates no longer holds. 
\end{rem}

{In the following proposition, by Proposition~\ref{prop: inver_ARK} and Remark~\ref{RGSA}, we can generalize the estimates (\ref{eq: lemma2_IMEX-E}) for the corresponding InDC method constructed with GSA IMEX R-K type A or type CK method, in the case of $\varepsilon = 0$, (reduced problem).}

\begin{prop}
\label{lemma1}
Consider that the reduced system (\ref{reduced}) with $\varepsilon=0$ satisfies {\eqref{eq: gz}} and with consistent initial values. {Consider the InDC IMEX R-K method constructed by GSA IMEX R-K methods of type CK or A.}
Then  the global error after $K$ correction loops satisfies the following estimates
\begin{eqnarray}
\label{eq: lemma2_local}
\begin{array}{l}
e^{(K)}_n \doteq {\hat{y}^{(K)}_n- y(t_n)} = \mathcal{O}(H^{\min(s_{K}, M)}) 
\quad
d^{(k)}_n \doteq {\hat{z}^{(K)}_n- z(t_n)} =\mathcal{O}(H^{\min(s_{K}, M)}),
\end{array}
\end{eqnarray} 
{with $\hat{y}^{(K)}_n$ and $\hat{z}^{(K)}_n$ being the numerical solution of the InDC methods at time  $t_n$ after $K$ corrections, $s_{K} = \sum_{k=0}^{K} p^{(k)}$, and $H = Mh$ is one InDC time step. The estimates hold uniformly for $H\le H_0$ and $nH \le Const$.}
\end{prop}
\begin{proof}
{From Proposition~\ref{prop: inver_ARK} and Remark \ref{rem: index1}  b), 
we get the estimates (\ref{eq: lemma2_local}).} 
\end{proof} 

Note that  if the IMEX R-K method is not GSA, the assembled matrix for the InDC-IMEX R-K method is not invertible and by Remark~\ref{rem: index1} and \ref{RGSA}, the error estimates (\ref{eq: lemma2_local}) in general are not satisfied and an order reduction for the case reduced problem (\ref{reduced}) ($\varepsilon = 0$) is observed.

%% file: main_results.tex
\section{Error estimates of InDC methods constructed with IMEX R-K}
\label{sec3}
\setcounter{equation}{0}
\setcounter{figure}{0}
\setcounter{table}{0}

In this section, we present the main theoretical result in the form of a theorem. On the contrary to what has been done in \cite{boscarino_qiu},
our idea here is to extend the error analysis of IMEX R-K methods applied to SPPs obtained in \cite{boscarino2008error}, to InDC methods constructed by IMEX R-K methods. In particular,  we  use the global error estimate results of IMEX methods for SPPs from \cite{boscarino2008error} to InDC-IMEX methods by the fact that an InDC-IMEX method can be viewed as an IMEX R-K method with assembled Butcher tableaus, as discussed 
in the previous section, to obtain global error estimates for InDC-IMEX methods.

Then, as in \cite{boscarino2008error}, in the main thorem, we estimate the {\em global} errors of the InDC-IMEX method
\[
e^{(K)}_{n, \nu} \doteq \hat{y}^{(K)}_{\nu}(t_n) - y_{\nu}(t_n), \quad  d^{(K)}_{n, \nu} \doteq \hat{z}^{(K)}_{\nu}(t^n) - z_{\nu}(t_n), \quad \nu = 0, 1
\]
where $\hat{y}^{(K)}_{\nu}(t_n)$ and $\hat{z}^{(K)}_{\nu}(t_n)$  are the $\nu$-th term in the $\eps$-expansion of the numerical solution of the InDC-IMEX method with the $K$ correction steps at some final time $t^n$.
Note that the estimates for the case $\nu = 0$ has been provided in the Proposition \ref{lemma1}.

\begin{thm} 
\label{thm: IDC_IMEX_R-K}
Consider the stiff system \eqref{spp}, \eqref{eq: gz} with  {well-prepared} initial values $y(0)$, $z(0)$ admitting a smooth solution. 
Consider the InDC method constructed with $M$ uniformly distributed quadrature nodes excluding the left-most point and a globally stiffly accurate IMEX R-K method of order $p^{(0)}$ {of type A or CK}. 
 Apply IMEX R-K methods of different classical orders $(p^{(1)}, p^{(2)}, \ldots,  p^{(K)})$ in the correction loops, $k=1, \cdots K$. Assume that each of these IMEX R-K methods in the correction loops are globally stiffly accurate. Then the global error after $K$ correction loops satisfies the following estimates
\beq
\begin{array}{lll}\label{final_estimate}
e^{(K)}_n \doteq \hat{y}^{(K)}(t_n) - y(t_n) &=& \mathcal{O}(H^{\min(s_{K}, M)})+\mathcal{O}(\eps H),\\
d^{(K)}_n \doteq \hat{z}^{(K)}(t_n) - z(t_n) &=& \mathcal{O}(H^{\min(s_K, {M})})+\mathcal{O}(\eps H),
\end{array}
\eeq
{where $\hat{y}^{(K)}(t_n)$ and $\hat{z}^{(K)}(t_n)$ are the numerical solutions of the InDC methods at $t^n$,}
for $\eps \le cH$ and for any fixed constant $c>0$, $s_{K} = \sum_{k=0}^{K} p^{(k)}$, and $H = Mh$ is one InDC time step. The estimates hold uniformly for $H\le H_0$ and $nH \le Const$. 
\end{thm}

Now we give the following proposition as a consequence of the Lemma 5.1 and theorem 5.2, and 6.2 in \cite{boscarino2008error}. The main Theorem \ref{thm: IDC_IMEX_R-K} follows from this result. 
\begin{prop}
\label{classP}
Consider a GSA IMEX R-K method of type A or CK, with $\tilde{b}_i \neq  b_i$, for $i = 1, \dots, s$, and let $p$ be the order of the explicit part of the scheme. Apply this method to the general problem (\ref{spp}), under the hypothesis (\ref{eq: gz}) with initial values consistent and such that the problem (\ref{spp}) admits a smooth solution. Then, for any fixed constant $C$, the global error satisfies for $\eps <Ch$
\begin{eqnarray}\label{1_0est}
{\hat{y}_{n}} - y(t_n) = \mathcal{O}(h^p) + \mathcal{O}(\eps h), \quad {\hat{z}_{n}} - z(t_n) = \mathcal{O}(h^p) + \mathcal{O}(\eps h),
\end{eqnarray}
These estimates hold uniformly for $h\le h_0$ and $nh<C$  for any fixed constant $C > 0$.
\end{prop}

\noindent
{\em Proof.} The proof of this Proposition is obtained by combining  Remarks \ref{rem: index1} a) b) and by the results in Theorem 5.2  and 6.2 in \cite{boscarino2008error}. 

\bigskip
\noindent
{\bf Proof of Theorem~\ref{thm: IDC_IMEX_R-K}.} {By Proposition \ref{prop: inver_ARK}, the InDC IMEX R-K method constructed by a GSA IMEX R-K method of type A or CK can be viewed as an IMEX R-K method of type A or CK and is GSA. Then the hypothesis of Proposition \ref{classP} are satisfied for the InDC IMEX R-K scheme. The estimates \eqref{final_estimate} are a consequence of Proposition \ref{lemma1} and \ref{classP}.}

We point out that the simplest InDC scheme is constructed by repeated use of the same IMEX scheme of order $p$, and the optimal choice of $M$ is given by $M = s_K = p (K+1)$.
This is our choice in the PDE applications presented in the paper.

%% file: numerical_evidence.tex
\section{Numerical evidence and PDE applications}
\label{sec5}
\setcounter{equation}{0}
\setcounter{figure}{0}
\setcounter{table}{0}

In this section, we consider the following InDC IMEX methods for stiff ODEs and PDEs.
{Below we present a list of IMEX R-K methods that are used in the InDC framework. These include first order GSA IMEX method of type A and ARS, and first order IMEX method of type A but not GSA. For high order IMEX methods, we choose methods of type CK or ARS.} We decide not to use high order GSA IMEX R-K methods of type $A$, because they are more expensive compared with methods of type CK and ARS. In fact, the construction of such methods involves many internal stages. For example for a second order GSA IMEX R-K method of type $A$, we require more than three ($s > 3$) internal stages. In fact, in \cite{boscarino2013implicit} the authors proved that it is not possible to construct second order GSA IMEX R-K methods of type A with three internal stages. 

\bigskip 
\noindent
{\bf InDC method embedded with a first order IMEX.}
\bit
\item We let InDC-IMEX1-GSA-ARS-M-k denote the InDC methods embedded with first order ($p = 1$) GSA IMEX R-K scheme of type ARS (\ref{first_schemeARS}) with $M$ quadrature points and $k$ correction steps. 
\item We let InDC-IMEX1-NGSA-M-k denote the InDC methods embedded with first order ($p = 1$) non globally stiffly accurate IMEX R-K scheme (\ref{first_schemeA}) with $M$ quadrature points and $k$ correction steps. The first order non globally stiffly accurate IMEX R-K method of type A (IMEX1-NGSA) has the following double Butcher tableau 
\item We let InDC-IMEX1-GSA-A-M-k denote the InDC method constructed with first order ($p = 1$) GSA IMEX R-K scheme of type A (\ref{first_schemeA_GSA}) with $M$ quadrature points and $k$ correction steps. 
\eit
{\bf InDC method embedded with a second order IMEX.}
\bit
\item
{We let InDC-IMEX2-ARS-M-k denote the InDC method constructed with the second order globally stiffly accurate IMEX R-K method of type ARS  with the following double Butcher tableau 
\begin{eqnarray}\label{SesondARS}
\begin{array}{c|ccc}
              0 & 0 & 0 & 0\\
              \gamma & \gamma & 0 & 0\\
           1 & \delta & 1-\delta & 0\\
               \hline
              &\delta &1-\delta & 0
\end{array} \qquad
\begin{array}{c|ccc}
                0 & 0 & 0 & 0\\
                \gamma & 0 &\gamma & 0  \\
               1 & 0& 1-\gamma & \gamma \\
              \hline 
               & 0 & 1-\gamma& \gamma
\end{array},
\end{eqnarray}
where $\gamma = 1-\frac{\sqrt{2}}{2}$ and $\delta = 1- 1/(2\gamma)$,
with $k$ correction steps and $M$ quadrature points.}
 This method has order $p=2$.
\item {We let InDC-IMEX2-CK-M-k denote the InDC method constructed with a second order globally stiffly accurate method of type CK with the following double Butcher tableau and $\gamma = 1-\sqrt{2}/2$
 \begin{eqnarray}\label{PtypeCK}
\begin{array}{c|ccc}
              0 & 0 & 0 & 0\\
              2/3 &  2/3 & 0&0\\
              1 & 1/4 & 3/4 & 0\\
              \hline
              & 1/4  & 3/4 & 0
\end{array} \qquad
\begin{array}{c|ccc}
0 & 0  & 0 & 0\\
               2/3 & 2/3 - \gamma & \gamma & 0\\
              1 & 1/4+\gamma/2& 3/4-3\gamma/2 & \gamma \\ 
              \hline 
              &1/4+\gamma/2& 3/4-3\gamma/2 & \gamma
\end{array} .
\end{eqnarray}
with $k$ correction steps and $M$ quadrature points.}
This method has order $p=2$.
\eit

\noindent
{\bf InDC method embedded with a third order IMEX.}
 \bit
\item We let InDC-IMEX3-ARS-M-k denote  the InDC method constructed with a third order globally stiffly accurate method of type ARS  with the following double Butcher tableau
\beq
\begin{array}{c|ccccc}
0& 0& 0& 0& 0& 0\\
1/2&1/2& 0& 0& 0&0\\
2/3& 11/18&1/18&0& 0& 0\\
1/2& 5/6& -5/6& 1/2& 0& 0\\
1& 1/4& 7/4& 3/4& -7/4& 0\\
\hline
&1/4& 7/4& 3/4& -7/4& 0\\
\end{array}. \ \ \ \ \ 
\begin{array}{c|ccccc}
0&0& 0& 0& 0& 0\\
1/2&0&1/2& 0& 0& 0\\
2/3&0&1/6&1/2& 0& 0\\
1/2&0& -1/2& 1/2& 1/2& 0\\
1& 0& 3/2& -3/2& 1/2& 1/2\\
\hline
& 0& 3/2& -3/2& 1/2& 1/2\\
\end{array}
\eeq
in the prediction and $k$ correction steps and $M$ quadrature points.
This method is stiffly accurate with order $p=3$.
\eit

The indicated order of convergence by Theorem~\ref{thm: IDC_IMEX_R-K} for the $y$ and $z$ components in the SPPs are summarized in Table~\ref{tab: error}. For the InDC-IMEX1-GSA-ARS-M-k and InDC-IMEX1-GSA-A-M-k methods, the order of convergence will increase with $k$ for the $\eps^0$ error term when $\eps \ll H$ and  $k\le M-1$, leading to a term of $H^{\min(k+1, M)}$ for the differential and algebraic component.
 Similar comments apply to the  InDC-IMEX2-ARS-M-k, InDC-IMEX2-CK-M-k, as well as the InDC-IMEX3-ARS-M-k with the order of accuracy for the first error term increased by $2$ or $3$ per correction step respectively. 
The order of convergence for the $\eps^1$ error term is $\eps H$ for all schemes we consider here. 
Note that for those InDC-IMEX methods with the same order of accuracy for the index 1 problem (i.e. when $\eps=0$), the complexity measured by the number of function evaluations is comparable. For example, when $M=8$, the number of function evaluations for the InDC-IMEX1-GSA-ARS-8-7, InDC-IMEX2-ARS-8-3, InDC-IMEX2-CK-8-3 are the same. Note that all of these methods achieve eighth order accuracy for the index 1 problem when $\eps=0$. {On the other hand, the number of function evaluations for InDC-IMEX1-GSA-A-8-7 is twice as much as InDC-IMEX1-GSA-ARS-8-7, e.g. see the number of stages in Butcher tableaus in \eqref{eq: semi-implicit-InDC_A} and \eqref{eq: semi-implicit-InDC_B}. From the computational cost point of view, the InDC-IMEX1-GSA-ARS method is preferred. For the sake of completeness, we present results for the InDC-IMEX1-GSA-A method for the ODE Van der Pol equation, but not for the PDE examples.}
 

\begin{table}[htb]
\begin{center}
\caption{{
Global error predicted by Theorem~\ref{thm: IDC_IMEX_R-K} with $H \gg \eps$. 
}\label{tab: error}
}
\bigskip
\begin{tabular}{|c | c|c|}
\hline
\cline{1-3} Method & $y-$comp &$z-$comp \\
\hline
\cline{1-3}  InDC-IMEX1-GSA-ARS-M-k & $H^{\min(k+1,M)} + \eps H$ & $H^{\min(k+1,M)} + \eps H$ \\
\hline
\cline{1-3}  InDC-IMEX1-GSA-A-M-k & $H^{\min(k+1,M)} + \eps H$ & $H^{\min(k+1,M)} + \eps H$ \\
\hline
\cline{1-3}   InDC-IMEX2-ARS-M-k& $H^{\min(2(k+1),M)} + \eps H$ & $H^{\min(2(k+1),M)} + \eps H$ \\
\hline
\cline{1-3}  InDC-IMEX2-CK-M-k  & $H^{\min(2(k+1),M)} + \eps H$ & $H^{\min(2(k+1),M)} + \eps H$ \\
\hline
\cline{1-3}  InDC-IMEX3-ARS-M-k  & $H^{\min(3(k+1),M)} + \eps H$ & $H^{\min(3(k+1),M)} + \eps H$\\
\hline
\end{tabular}
\end{center}
\end{table}

\subsection{Van der Pol example} 
\label{sec: 5.1}

For numerical verification, we test a standard nonlinear oscillatory
test problem, Van der Pol�s equation with well-prepared initial data up to $\mathcal{O}(\eps^3)$,  \cite{hairer1993solving2}:
\beq
\left\{
\begin{array}{l}
y' = z \\
\eps z' = (1-y^2) z - y
\end{array}
\right.,
\quad
\left\{
\begin{array}{l}
y(0) = 2 \\
z(0) = -\frac23 + \frac{10}{81}\eps -\frac{292}{2187} \eps^2
\end{array}
\right.
\eeq
and $\varepsilon = 10^{-6}$.

Numerical observations in Figures \ref{fig1} and \ref{fig2} are consistent with Theorem~\ref{thm: IDC_IMEX_R-K} and Table~\ref{tab: error}. They produce estimates for the $y$ and $z$ component in the form of equation \eqref{final_estimate}. Especially, numerical results of InDC method constructed with first order IMEX methods presented in Figure~\ref{fig1} and of InDC method constructed with high order (second or third order) InDC-IMEX2-ARS-M-K presented in Figure~\ref{fig2}, confirm the theoretical prediction \eqref{final_estimate}, i.e., if the IMEX method is GSA, the order of convergence for $\eps^0$ term for the $y$ and $z$ component increases with the correction loops (that is $s_K$ with $K$ the number of correction loops). However, such improvement for $\eps^0$ term is not true if the IMEX method is not globally stiffly accurate (see two panels in the top row of Figure~\ref{fig1} for InDC-IMEX1-NGSA method). 
Furthermore, in the estimate \eqref{final_estimate}, the InDC-IMEX methods exhibit order reduction both in differential and algebraic components (see Table~\ref{tab: error} for every type of InDC method). This phenomenon appears, since the $\eps^1$ term of the error behaves like $\mathcal{O}(\eps H)$  in \eqref{final_estimate} for both $y$ and $z$ components (see in Figures \ref{fig1} and \ref{fig2}). When $H$ is very small, the $\mathcal{O}(\varepsilon H)$ term is dominant in the estimates of $y$ and $z$-components respectively. {Finally, we want to comment that, despite the same order of convergence, the magnitude of errors of InDC-IMEX1-GSA-A-7-3 are much smaller than those of InDC-IMEX1-GSA-ARS-7-3. As we mentioned earlier, the InDC-IMEX1-GSA-ARS method costs less number of function evaluations and is more efficient.
} 

\begin{figure}
\centering
\includegraphics[width=2.5in]{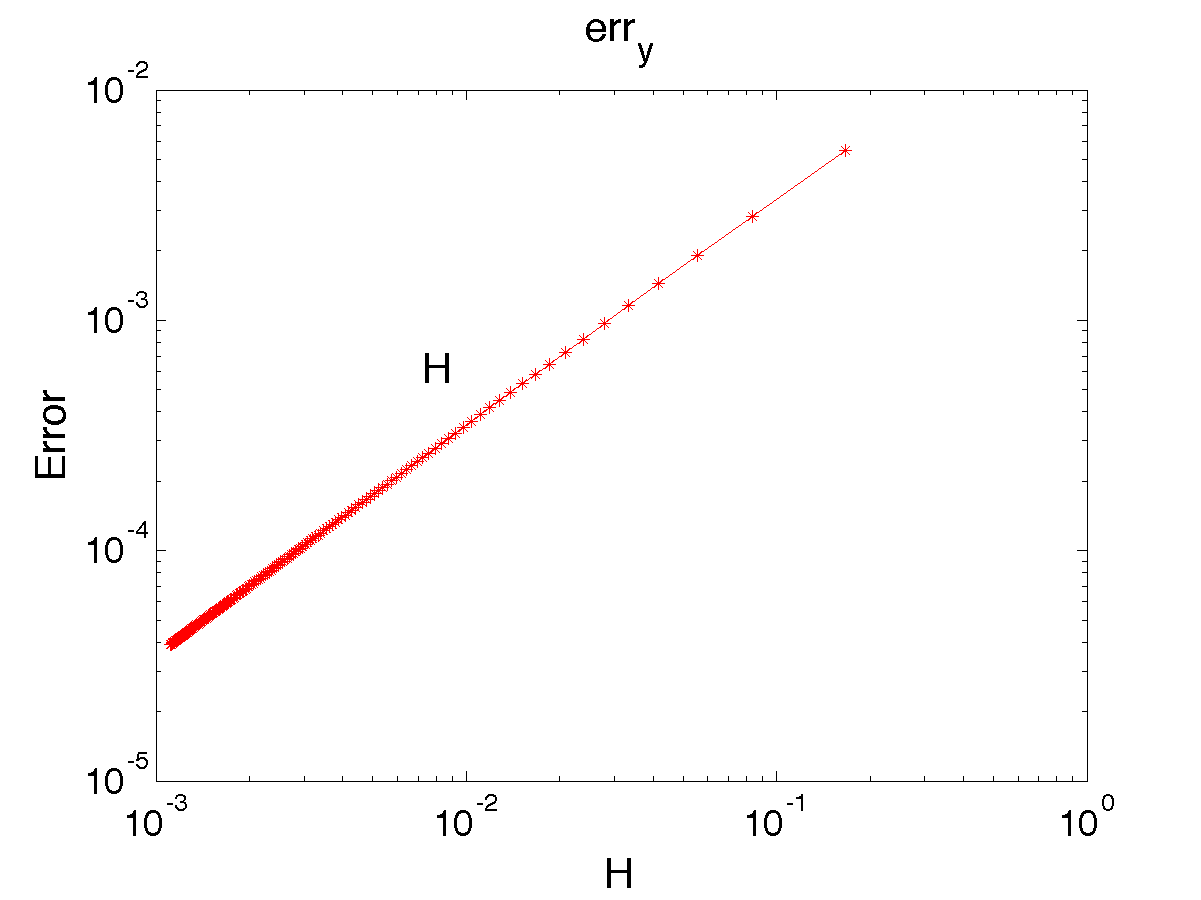},
\includegraphics[width=2.5in]{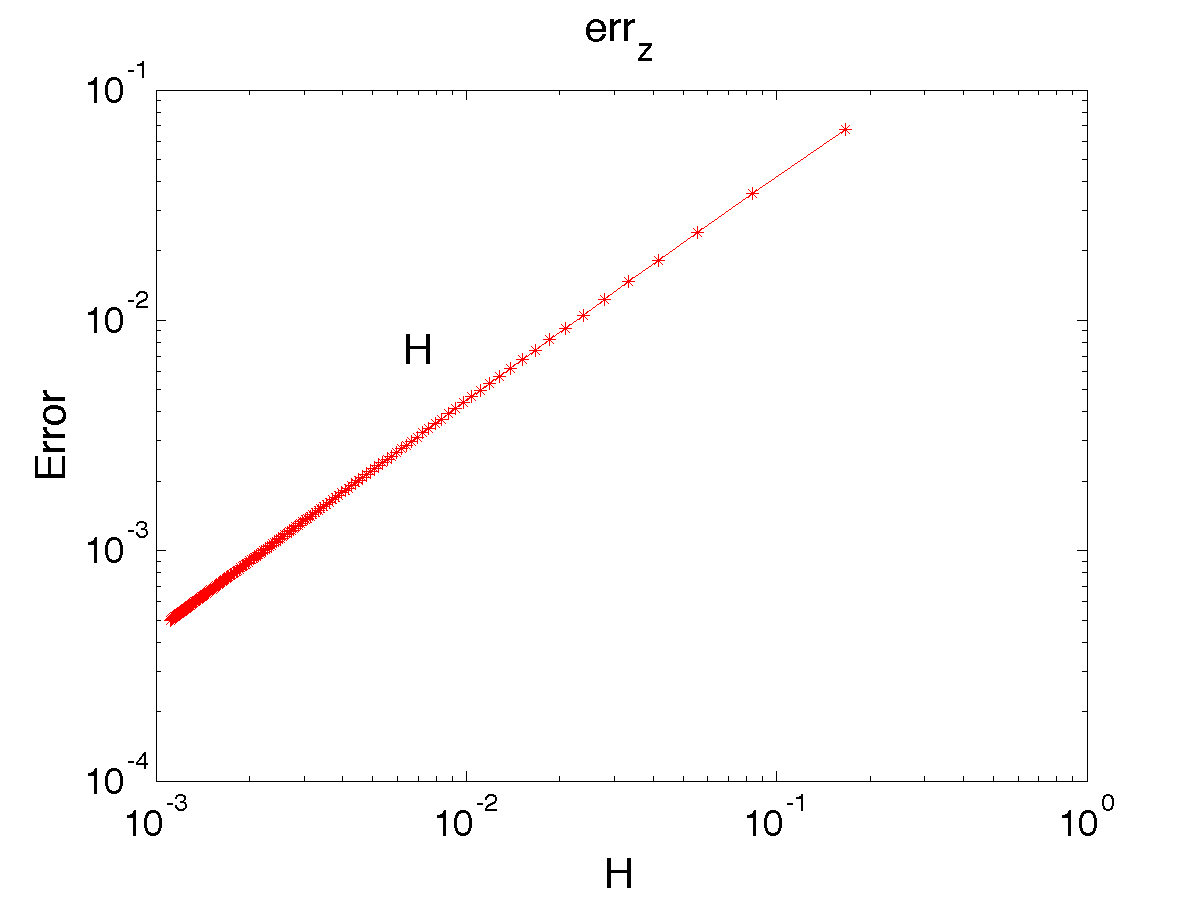}
\centering
\includegraphics[width=2.5in]{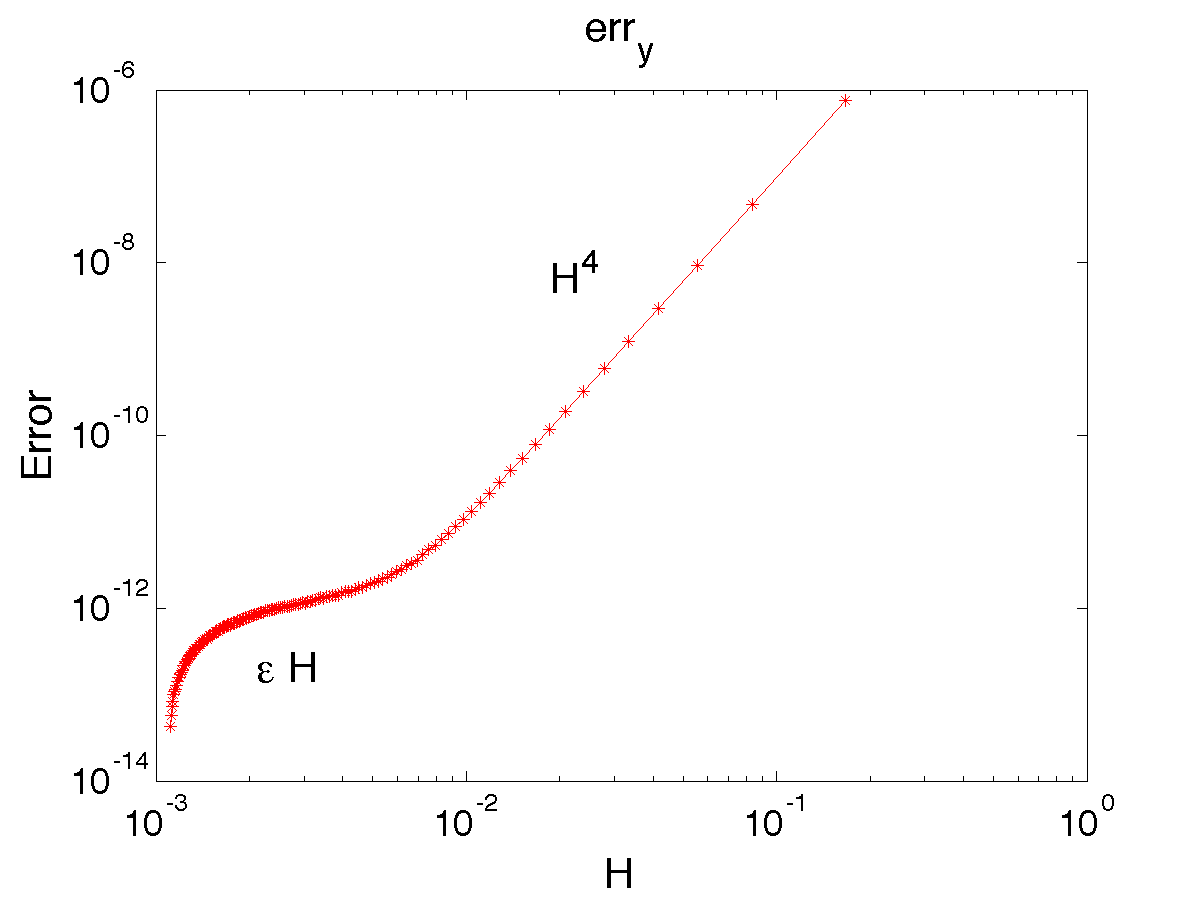},
\includegraphics[width=2.5in]{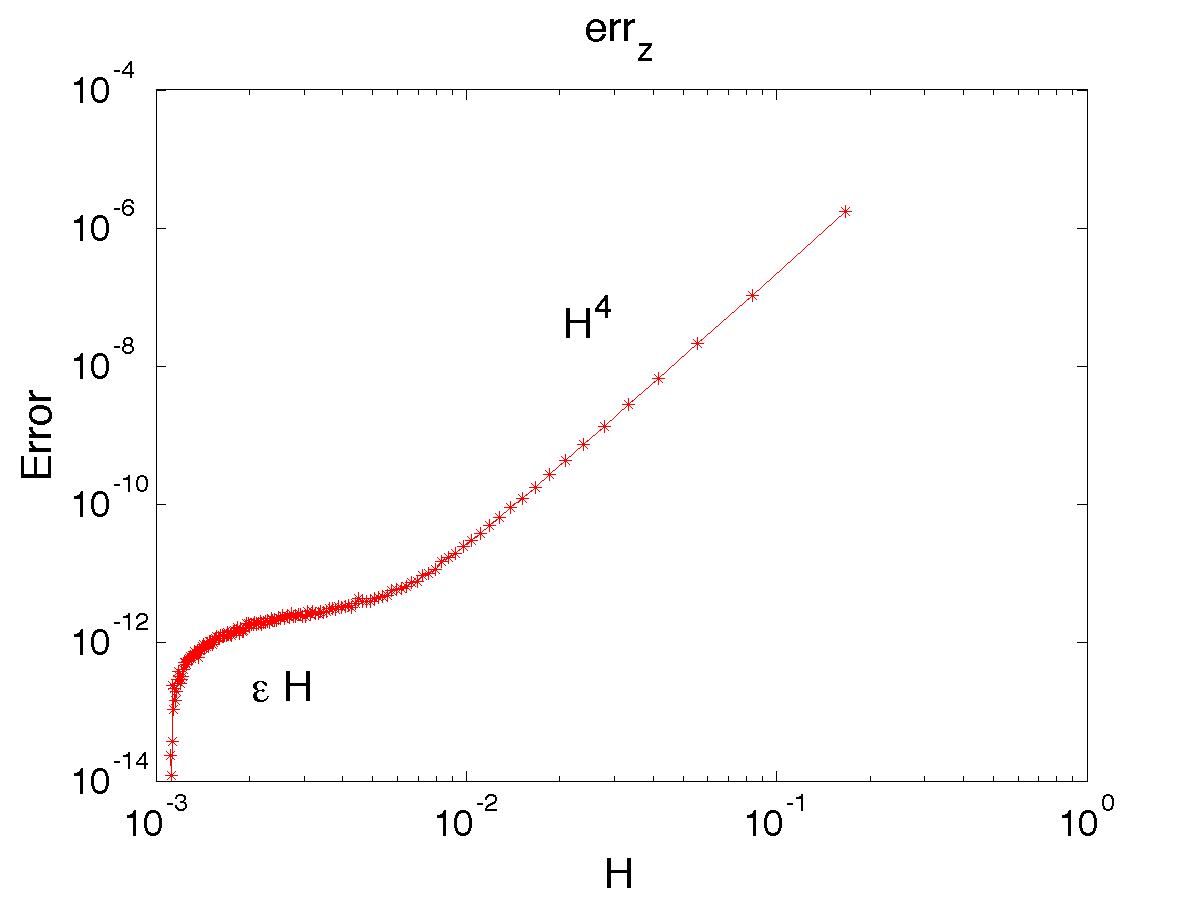}
\centering
\includegraphics[width=2.5in]{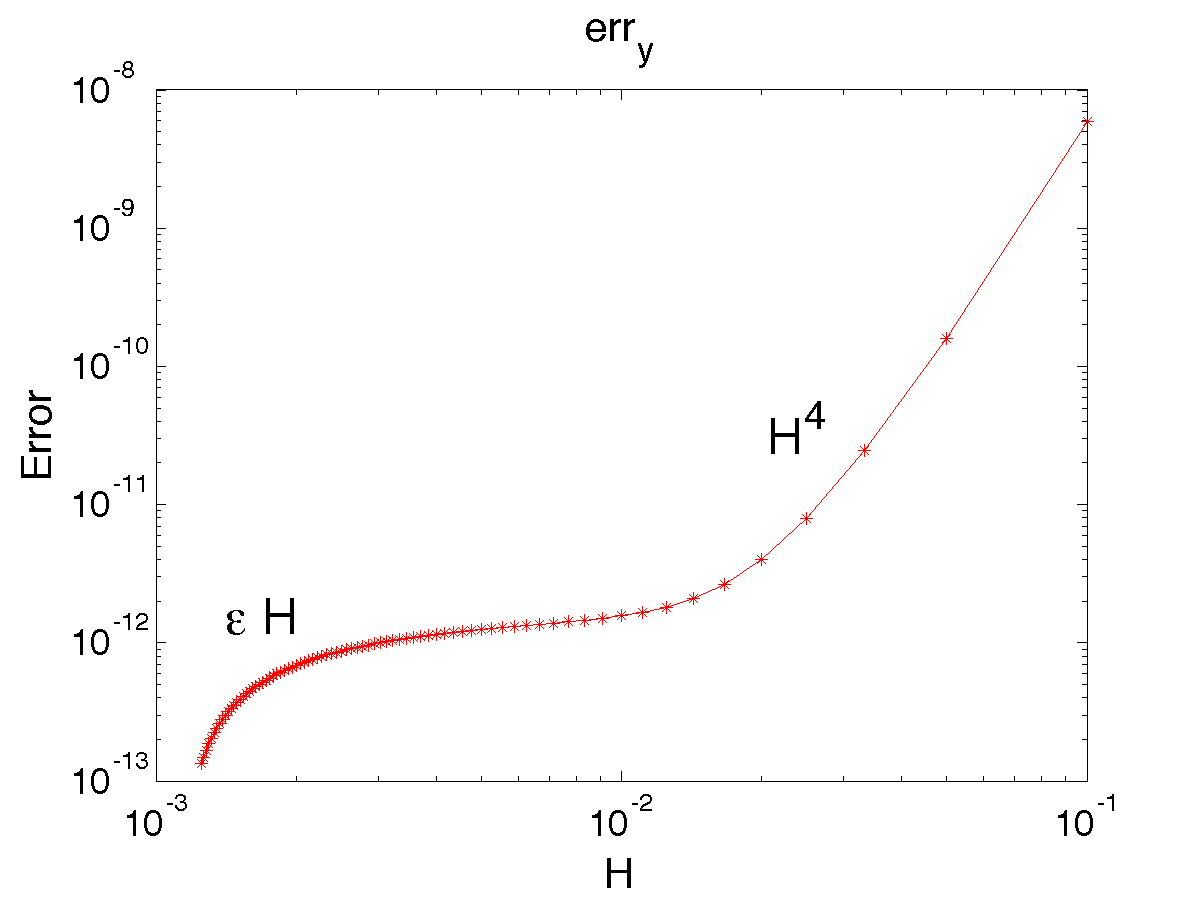},
\includegraphics[width=2.5in]{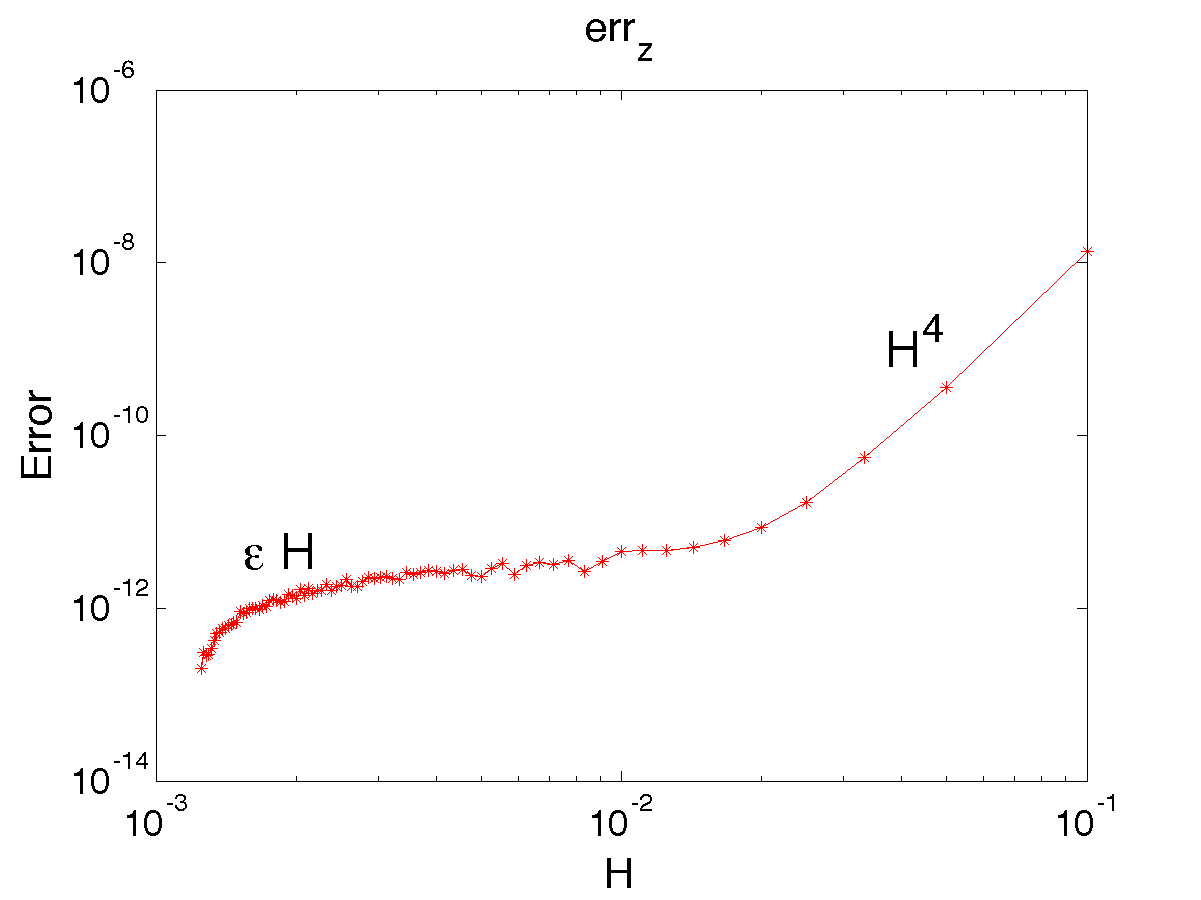}
\caption{
Van der Pol equation.
Global error ($T = 0.5$) of the InDC-IMEX1-NGSA-5-3 method (top row), InDC-IMEX1-GSA-ARS-7-3 method (middle row),
and InDC-IMEX1-GSA-A-7-3 method (bottom row).
$\eps = 10^{-6}$. 
}
\label{fig1}
\end{figure}

\begin{figure}
\centering
\includegraphics[width=2.5in]{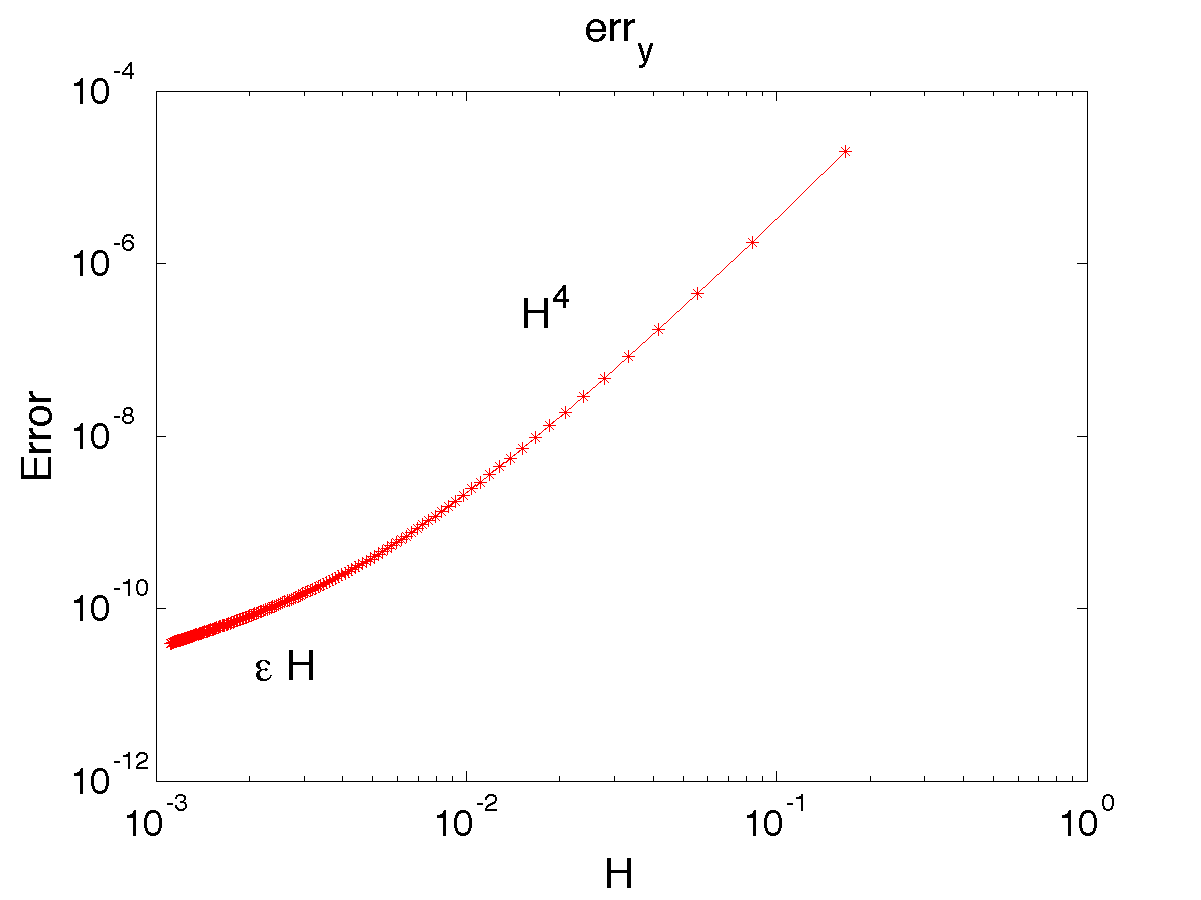},
\includegraphics[width=2.5in]{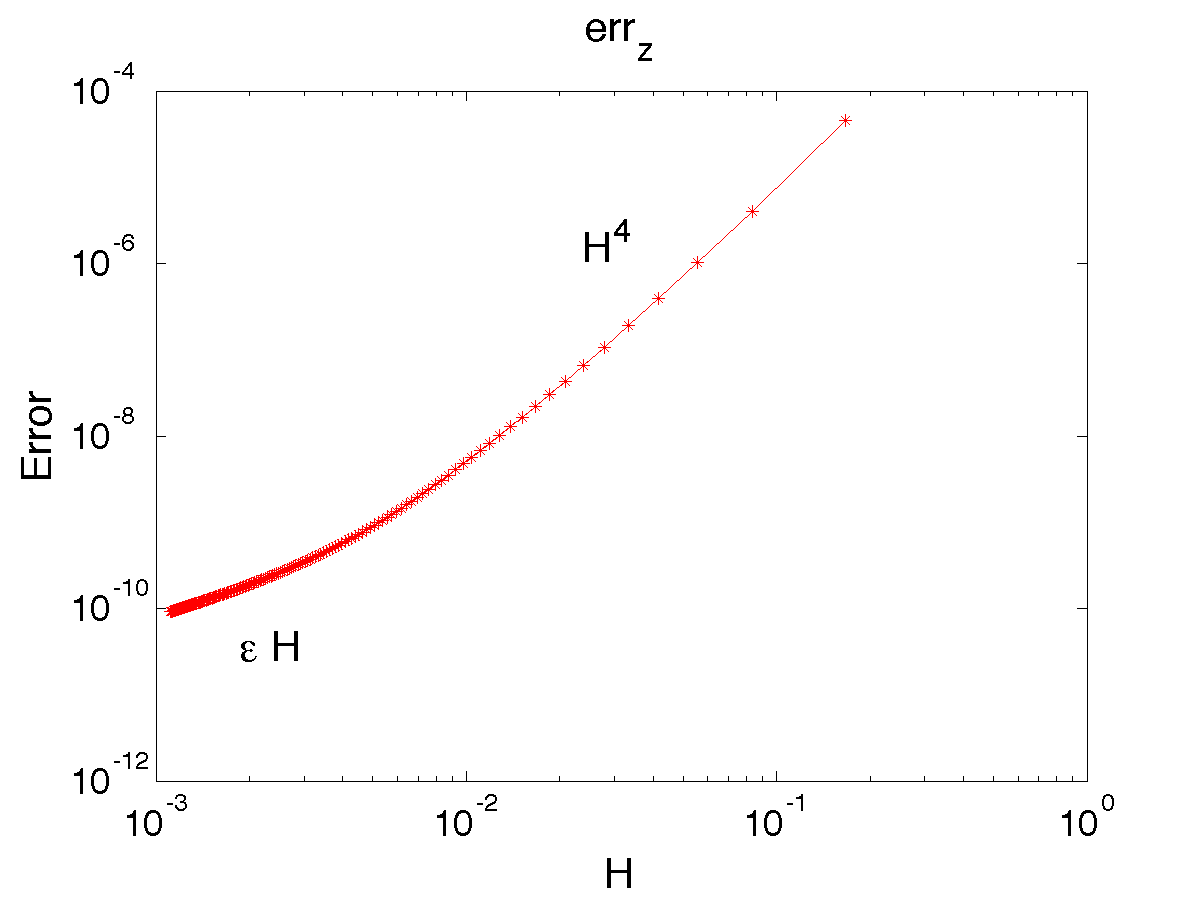}
\centering
\includegraphics[width=2.5in]{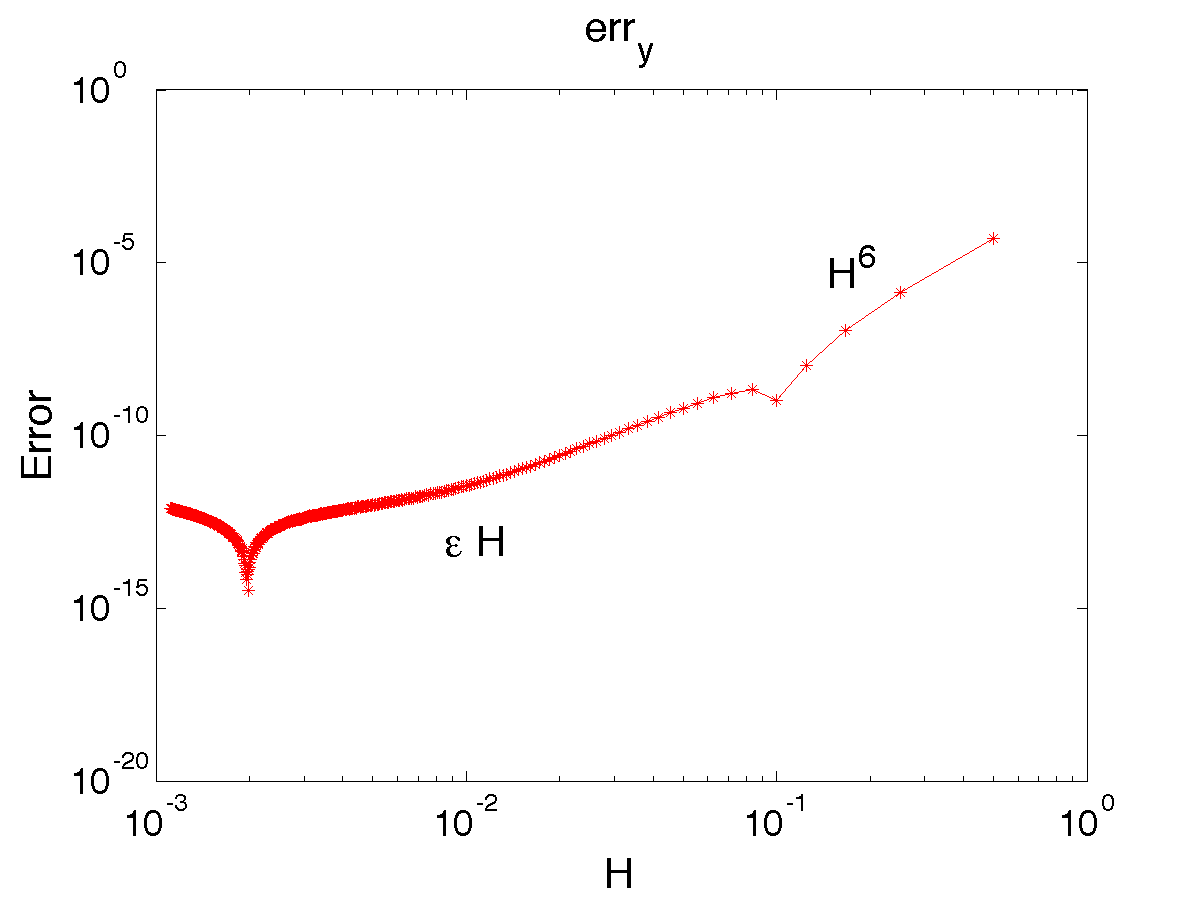},
\includegraphics[width=2.5in]{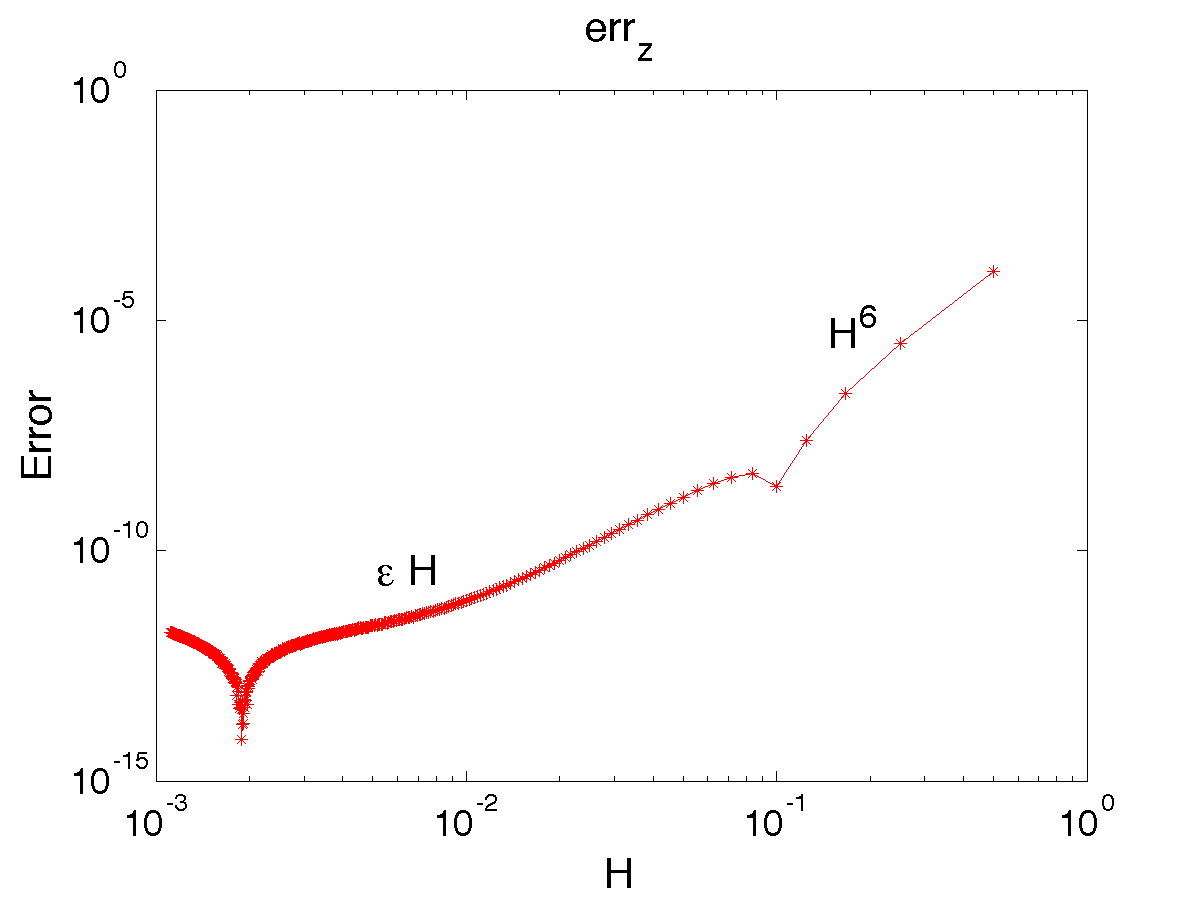}
\centering
\includegraphics[width=2.5in]{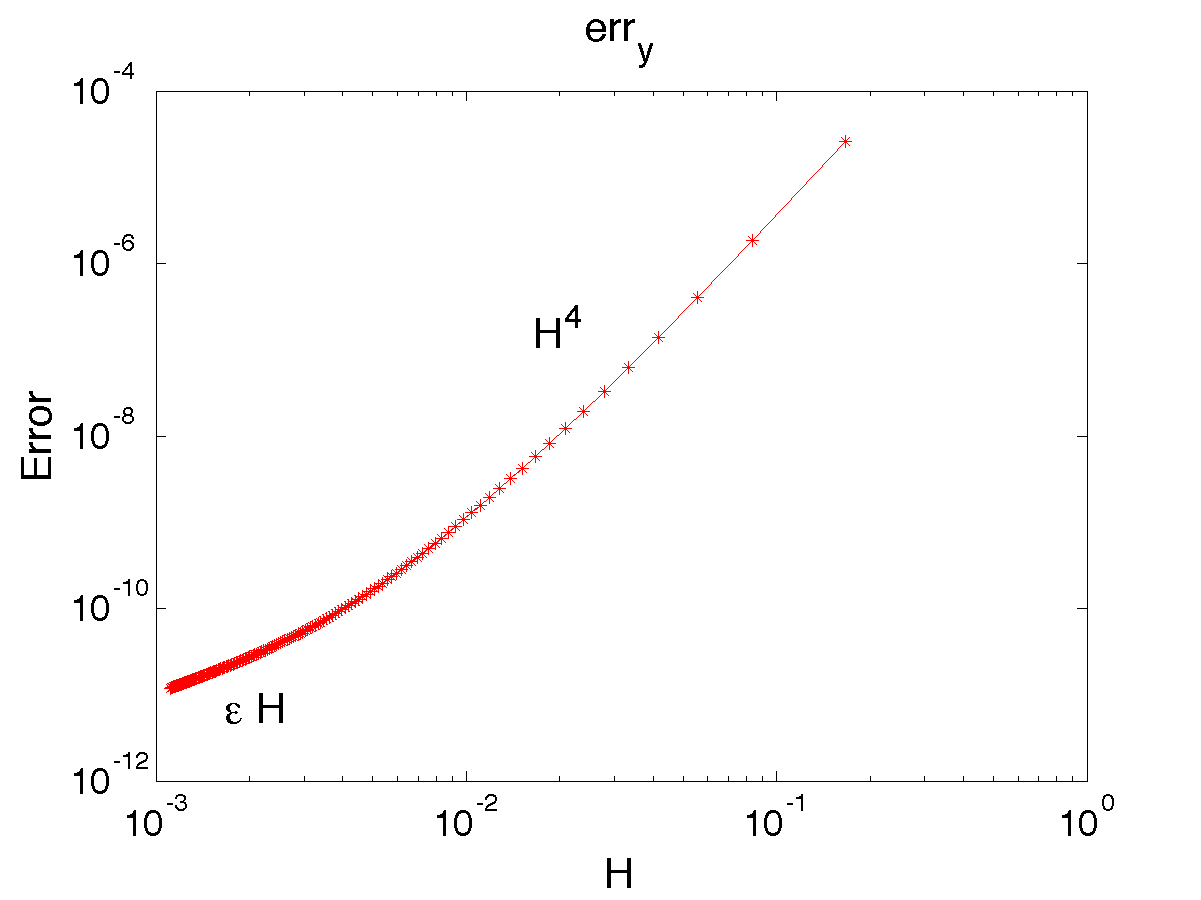},
\includegraphics[width=2.5in]{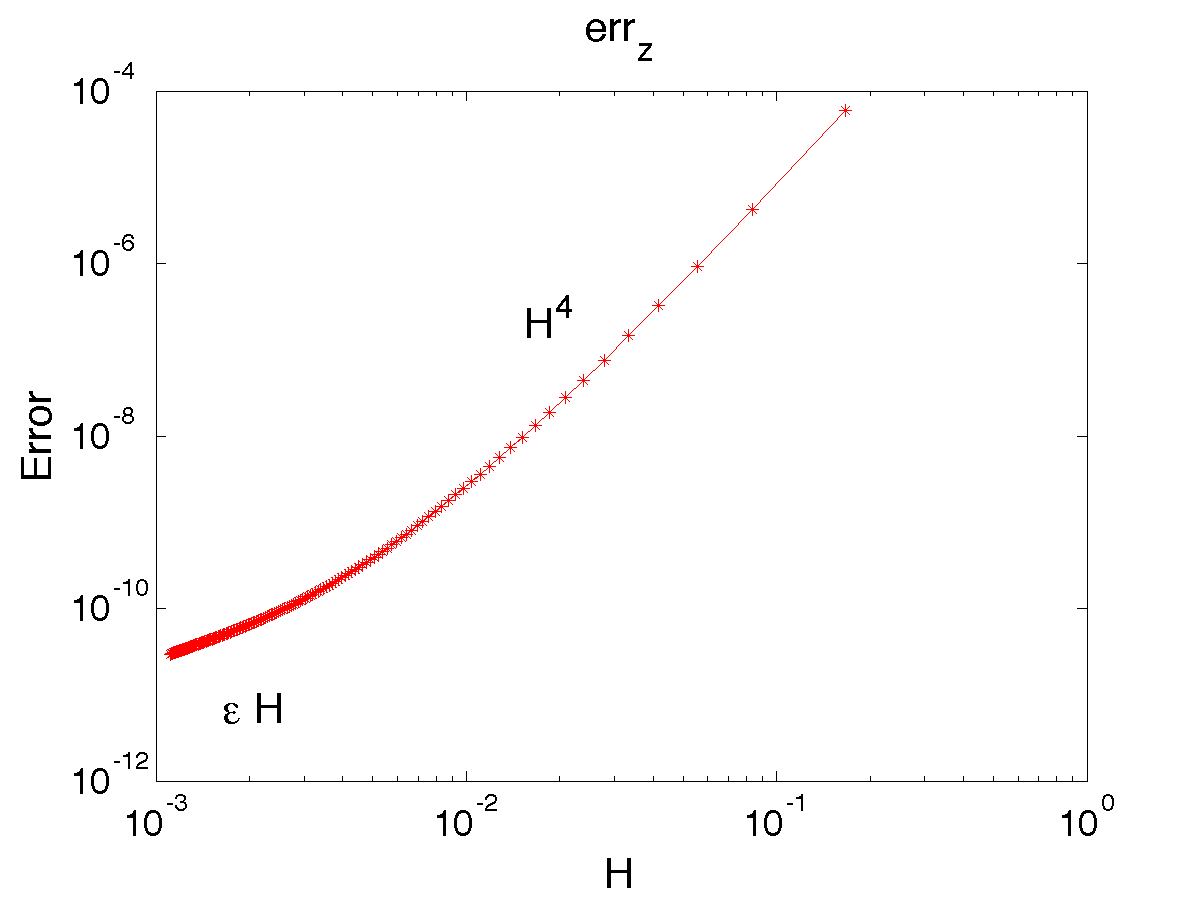}
\centering
\includegraphics[width=2.5in]{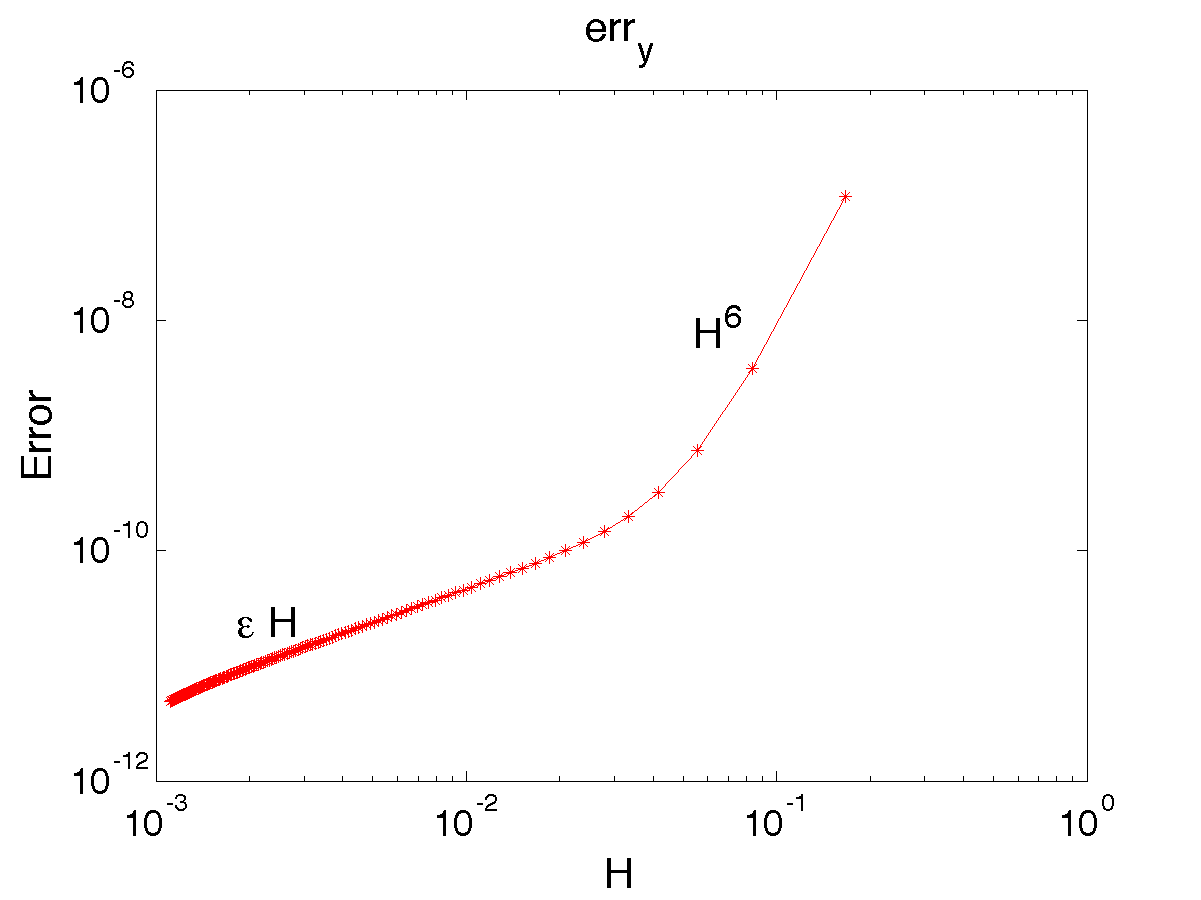},
\includegraphics[width=2.5in]{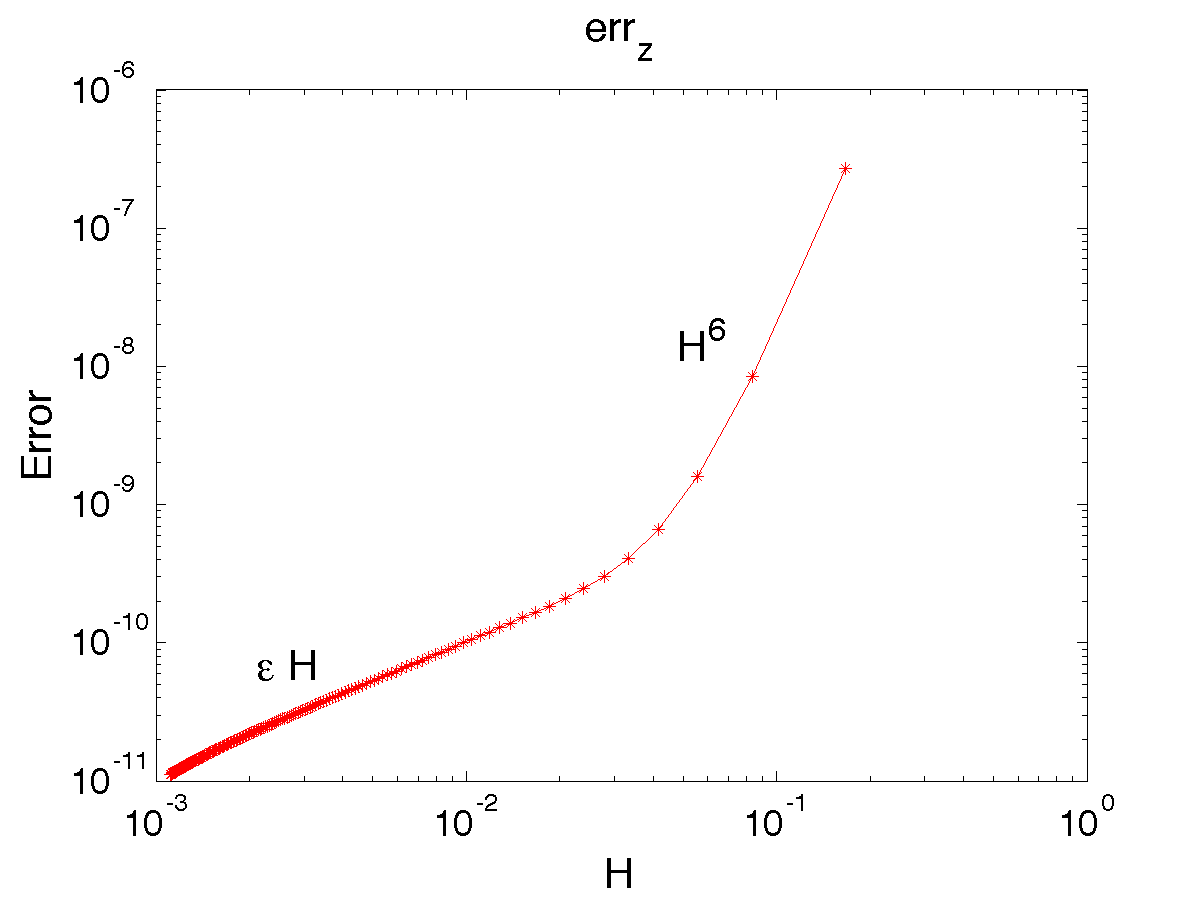}
\caption{
Van der Pol equation.
Global error ($T = 0.5$) of the InDC-IMEX2-ARS-5-1 method (upper row),
of the InDC-IMEX2-ARS-7-2 method (the second row),
of the InDC-IMEX2-CK-5-1 method  (the third row),
and of the InDC-IMEX3-ARS-7-1 method  (bottom row).
$\eps = 10^{-6}$.
}
\label{fig2}
\end{figure}

\subsection{PDEs examples}

In this section we consider systems of the form (\ref{start}), which may arise from a method of lines discretization of PDEs. Notice that in PDEs applications, such
as advection-diffusion problem, convection-diffusion-reaction systems and hyperbolic systems with relaxations, the structure of the problems is
usually of additive form as (\ref{start}).

\begin{exa} \label{BurgTest} ({\bf Advection-diffusion equation})
As an example of advection-diffusion equation, we consider the 1D viscous Burgers' equation,
\beq\label{Burg}
u_t + \left(\frac{u^2}{2}\right)_x = \frac{1}{R}u_{xx}, \quad R>0. 
\eeq
It contains a non-linear advection term $(u^2/2)_x$ and a dissipation term $u_{xx}/R$, where $R$ is called the \emph{Reynolds number}. We remark that, when the Reynolds number is large, the convection term dominates the diffusion term. The solution develops a sharp shock wave front after a certain time of propagation and when $R\to \infty$ discontinuities appear. In this test problem we choose $R = 0.1$ so that the dissipation term dominates the advection one and a smooth solution is produced. 

The system (\ref{Burg}), after a method of lines approach with $U_i (t) = u( x_i, t)$ and $U(t) = (U_1(t), U_2(t), \cdots, U_N(t))^T$, has the form of (\ref{start}). The convection term $-(u^2/2)_x$ is considered as non-stiff and is integrated using the explicit part of the method; it is discretized by high order finite difference discretization with Weighted-Essentially Non Oscillatory (WENO) reconstruction \cite{shu1998essentially}.
The diffusion term $u_{xx}$ is considered stiff and is treated implicitly; it is discretized by the standard fourth order finite difference technique, excepted at the nearby boundary points where a 3-rd order formula was implemented.

To check the temporal order of convergence, we consider the periodic boundary condition and a smooth initial condition $u(x, 0) = \sin(\pi x)$ with $x\in[0, 1]$. The equation has been integrated to time $T = 0.02$. Notice that the solution drops dramatically from $sin(\pi x)$ to zero within the first $0.05$ second. The exact solution of this problem was obtained by Cole \cite{cole1951quasi} so that numerical comparison can be done.
In the Tables~\ref{tab: ConvRate2} and ~\ref{tab: ConvRate}, we show the $L^2$ error and the corresponding order of convergence in time ``Order''  for a sequence of mesh sizes $\Delta t =\Delta x= 1/N$.
We use the InDC-IMEX1 method with one or two corrections, corresponding to a second or third order time integration scheme and InDC-IMEX2-ARS-GSA-4-2 method with one correction, corresponding to a fourth order scheme. The expected temporal order of convergence is numerically observed. 

\begin{table}[htb]
\begin{center}
\caption{{Example \ref{BurgTest}. Accuracy test with the Burgers' equation
 for InDC- IMEX1-M-k GSA, with k = 1, 2 corrections.
}\label{tab: ConvRate2}
}
\bigskip
\begin{tabular}{|c | c|c|c|}
\hline
\cline{1-3} Method & $\Delta t$& $L_2$ error &Order \\
\hline
\cline{1-3}   InDC-IMEX1-2-1 & 1/96&$2.4551e-02$ & $--$ \\
\hline
\cline{1-3}   InDC-IMEX1-2-1 & 1/192& $9.0755e-03$ & $1.435$ \\
\hline
\cline{1-3}    InDC-IMEX1-2-1   &{1/384}& $2.7988e-03$ & $1.697$ \\
\hline
\cline{1-3}   InDC-IMEX1-2-1 &{1/768} &$7.9568e-04$ & $1.814$\\
\hline
\cline{1-3}   InDC-IMEX1-2-1  &{1/1536} &$2.1206e-04$ & $ 1.907$\\
\hline
\cline{1-3}   InDC-IMEX1-3-2  & {1/3072}& $5.4748e-05$ & ${1.953}$\\
\hline
\hline
\hline
\cline{1-3}   InDC-IMEX1-3-2 & 1/96& $6.9529e-04$ & $--$ \\
\hline
\cline{1-3}   InDC-IMEX1-3-2& 1/192& $2.2174e-04$ & $1.648$ \\
\hline
\cline{1-3}    InDC-IMEX1-3-2  & 1/384& $4.1888e-05$ & $2.404$ \\
\hline
\cline{1-3}   InDC-IMEX1-3-2 & 1/768& $6.5849e-06$ & $2.669$\\
\hline
\cline{1-3}   InDC-IMEX1-3-2  & 1/1536& $9.1527e-07$ & $2.846$\\
\hline
\cline{1-3}   InDC-IMEX1-3-2  & {1/3072}& $1.2111e-07$ & ${2.920}$\\
\hline
\end{tabular}
\end{center}
\end{table}
\begin{table}[htb]
\begin{center}
\caption{{
Example \ref{BurgTest}. Accuracy test with the Burgers' equation example
for second order IMEX2-ARS GSA method and InDC-IMEX2-ARS-GSA-M-k, with k = 1 correction.
}\label{tab: ConvRate}
}
\bigskip
\begin{tabular}{|c | c|c|c|}
\hline
\cline{1-3} Method & N & $L_2$ error & Order \\
\hline
\cline{1-3}  InDC-IMEX2-ARS-GSA-4-2 &1/48& $7.6079e-04$ & $--$ \\
\hline
\cline{1-3}   InDC-IMEX2-ARS-GSA-4-2 &1/96& $2.7095e-05$ & $4.811$ \\
\hline
\cline{1-3}   InDC-IMEX2-ARS-GSA-4-2  &1/192& $1.1064e-06$ & $4.614$ \\
\hline
\cline{1-3}   InDC-IMEX2-ARS-GSA-4-2  &1/384& $5.5847e-08$ & $4.302$\\
\hline
\cline{1-3}   InDC-IMEX2-ARS-GSA-4-2  &1/758& $3.4640e-09$ & $4.010$\\
\hline
\cline{1-3}   InDC-IMEX2-ARS-GSA-4-2  &1/1536 & $2.3912e-10$ & $3.856$\\
\hline
\end{tabular}
\end{center}
\end{table}
\end{exa}


Next we consider hyperbolic systems with stiff relaxation.  These systems have the form
\beq \label{HypMain}
\begin{array}{l}
\mathcal{U}_t + \mathcal{F}(\mathcal{U})_x = \frac{1}{\eps}\mathcal{G}(\mathcal{U}),
\end{array}  
\eeq
where $\mathcal{U} \in \mathbb{R}^N$, $\mathcal{F, G}: \mathbb{R}^N \to \mathbb{R}^N$ and the Jacobian matrix $\mathcal{F}'(\mathcal{U})$ has real eigenvalues and admits
a basis of eigenvectors for all $\mathcal{U} \in \mathbb{R}^N$ and $\varepsilon  > 0$ is the stiffness parameter. The operator $\mathcal{G}: \mathbb{R}^N \to \mathbb{R}^N$ is called a relaxation operator. (\ref{HypMain}) defines a relaxation
system. For mathematical results concerning this kind of problems we refer to \cite{liu1987hyperbolic} and \cite{chen1994hyperbolic}.
Same as for the previous example, we discretize the system (\ref{HypMain}) by a finite difference scheme with $U_i (t)$ approximates the solution at grid point $x_i$, $i=1, \cdots N$ and 
$U(t) = (U_1(t), U_2(t), \cdots, U_N(t))^T$ be solutions at all grid points. With the method of line approach, $U(t)$ satisfies a system in the form of 
(\ref{start}), where the function $F({U})$ in \eqref{start} as a discretization of the convective term $-\mathcal{F}(\mathcal{U}))_x$ is being treated explicitly, and $G(U)$ as a discretization of the source term $\mathcal{G}(\mathcal{U})$ is being treated implicitly.

A prototype example that we will use to illustrate our theoretical findings is the following $2\times 2$ nonlinear hyperbolic system with relaxation
(\cite{boscarino2010class}, \cite{pareschi2005implicit}, \cite{chen1994hyperbolic}) 
\beq \label{Hyp1}
\begin{array}{l}
u_t + f_1(u,v)_x  = 0, \\
v_t + f_2(u, {v})_x = \frac{1}{\eps}g(u,v).
\end{array}  
\eeq
System (\ref{Hyp1}) is a particular case of (\ref{HypMain}) with $\mathcal{U} = (u, v)^T$ , $u \in \mathbb{R}^n$, $v \in \mathbb{R}^{N-n}$, $n<N$, $\mathcal{F} = (f_1, f_2)^T$
and $\mathcal{G} = (0, g)^T$. 
Note that system (\ref{Hyp1}) possesses a unique local equilibrium, namely, $g(u,v) = 0$ implies $v = q(u)$ and, at the local equilibrium, one has the macroscopic
system
\beq \label{Hyp1_2}
\begin{array}{l}
u_t + f_1(u, q(u))_x = 0.
\end{array}  
\eeq
Equation (\ref{Hyp1_2}) can be derived by sending $\varepsilon$ in (\ref{Hyp1}) to zero, the so-called zero relaxation limit, (\cite{boscarino2010class}, \cite{pareschi2005implicit}, \cite{chen1994hyperbolic}).
If we use a method of lines to (\ref{Hyp1}), we get a ODE system of the form (\ref{start}). 
\begin{exa}\label{HSwR}
Consider a linear system with stiff relaxation source term
\beq \label{Hyp}
\begin{array}{l}
\partial_t u + \partial_x v  = 0, \\
\partial_t v + \partial_x u = -\frac{1}{\eps}(v - bu), \quad x \in [0, 2], \quad t \in [0, T],
\end{array}  
\eeq
with $\eps > 0$ and $b$ constant. Here $\mathcal{F}(\mathcal{U}) = (v, u)^T$ and $\mathcal{G}(\mathcal{U}) = (0, (v-bu))^T$.  

{To gain an understanding of the system (\ref{Hyp}), we} 
consider a formal expansion of solutions in the form
\beq \label{expHyp}
\begin{array}{l}
u = u_0  + u_1 \varepsilon + \mathcal{O}(\varepsilon^2), \\
v = v_0  + v_1 \varepsilon + \mathcal{O}(\varepsilon^2), \\
\end{array}  
\eeq
and insert it in (\ref{Hyp}).
Collect leading terms in approximating the second equation of (\ref{Hyp}) (order $\varepsilon^0$), we get $v_0 = bu_0$, plug which into the first equation, we obtain $\partial_t u_0 + b\partial_x u_0 = 0$. This equation is formally obtained as $\varepsilon \to 0$, and we call it \emph{reduced} equation. 
Next, we consider the first-order correction to the leading term approximation. By looking for the order $\varepsilon$ correction to the approximation (\ref{Hyp}) we get 
{Then, by keeping first-order terms in the expansion (\ref{expHyp}) and neglecting second and higher
order terms, we obtain a dissipative evolutionary equation
\beq 
\begin{array}{l}
v = bu - (1-b^2)\partial_x u\\
\partial_t u + b \partial_x u = \varepsilon \partial_x((1-b^2)\partial_xu).
\end{array}
\eeq
}
The second equation is a convection-diffusion equation with viscosity coefficient $\nu = \varepsilon (1 - b^2)$ and it is dissipative if $|b| \le 1$ (subcharacteristc
condition of Liu \cite{chen1994hyperbolic} for (\ref{Hyp})). 

Then motivated by this analysis we perform an accurate test for this problem considering well-prepared initial data. Note that the initial conditions for the coefficients of $u_i(0)$ in (\ref{expHyp}) can be chosen arbitrarily, but there is no freedom in the choice of $v_i(0)$. We let $u_0(x, 0) = \sin(2\pi x)$, and $u_1(x, 0) = 0$ with $v_0(x, 0) = b u_0(x, 0)$ and $v_1(x, 0) = (b^2 - 1)\partial_x u(x, 0)$. Such initial data is consistent, i.e., $g(u_0(x,0), v_0(x,0)) = (bu_0(x,0) - v_0(x,0)) = 0$ for $\varepsilon = 0$.

We consider a periodic smooth solution and we set $b = 0.5$ and $T= 0.2$. The spatial discretization of the domain has a fixed space mesh $\Delta x = 0.02$.
Our test problems are computed with coarse temporal grids that do not resolve small scales. High accuracy in space is obtained by finite difference discretization with WENO reconstruction \cite{shu1998essentially}. {To compute the error, the exact solution is obtained by the Fourier series}
\[
U_{ex}(x,t) = \sum^{+ \infty}_{\ell = -\infty} U_\ell(t)e^{i\ell x}, \quad V_{ex}(x,t) = \sum^{+ \infty}_{\ell= -\infty} V_\ell(t)e^{i\ell x}
\]
with $U_\ell(t)$ and $V_\ell(t)$ satisfying the ODE system
\begin{eqnarray}\label{2x2}
\begin{array}{l}
\dot{U}_\ell = -i\ell V_\ell,\\
\dot{V}_\ell = -i\ell U_\ell - \frac{1}{\eps}(V_\ell - bU_\ell) 
\end{array}
\end{eqnarray}
For each $\ell$, system (\ref{2x2}) can be written as a 2x2 constant coefficient homogeneous system which can be solved exactly. {For the initial condition $u(x, 0) = \sin(2\pi x)$ and the corresponding one for $v$, it is sufficient to consider $\ell=1$ only to obtain the exact solution.}

In Figure \ref{fig5}, {we plot the temporal convergence rate of the $u$-component as a function of $\eps$.} 
Here order of accuracy curves are obtained in the following way: on each curve and for each value of $\varepsilon$, the order of accuracy is computed from measuring $L^1$ errors {computed by the difference between the exact solution and the numerical one} for two different time step sizes $(\Delta t, \Delta t/2)$. Expected high order convergence rate is observed in the limiting cases of $\varepsilon \to 0$ and $\varepsilon \approx 1$ for the InDC-IMEX1-M-k when $M=2, 3, 4$ and $k = 1, 2, 3$ 
{and for the InDC-IMEX2-ARS-GSA-M-k when $M= 4$ and $k = 0, 1$ respectively.}
We observe that the convergence rate is increased by InDC correction iterations for sufficiently stiff parameters; however, for intermediate values of the parameter $\varepsilon$, e.g., $10^{-4} < \varepsilon < 10^{-2}$, we have a deterioration of the accuracy, as we expect. This is an indication that these InDC IMEX methods suffer from the phenomenon of order reduction in the mildly stiff regime ($\Delta t = \mathcal{O}(\varepsilon$)), when the classical order is greater than two \cite{boscarino2008error, boscarino2009accurate, CarpenterKennedy}. From the practical point of view, the understanding of this phenomenon is essential in situations where one is interested in the construction of higher order methods.
{In the next example we will give a brief explanation of this lack of accuracy for mildly stiff regime.}
\begin{figure}
\centering
\includegraphics[width=3.2in]{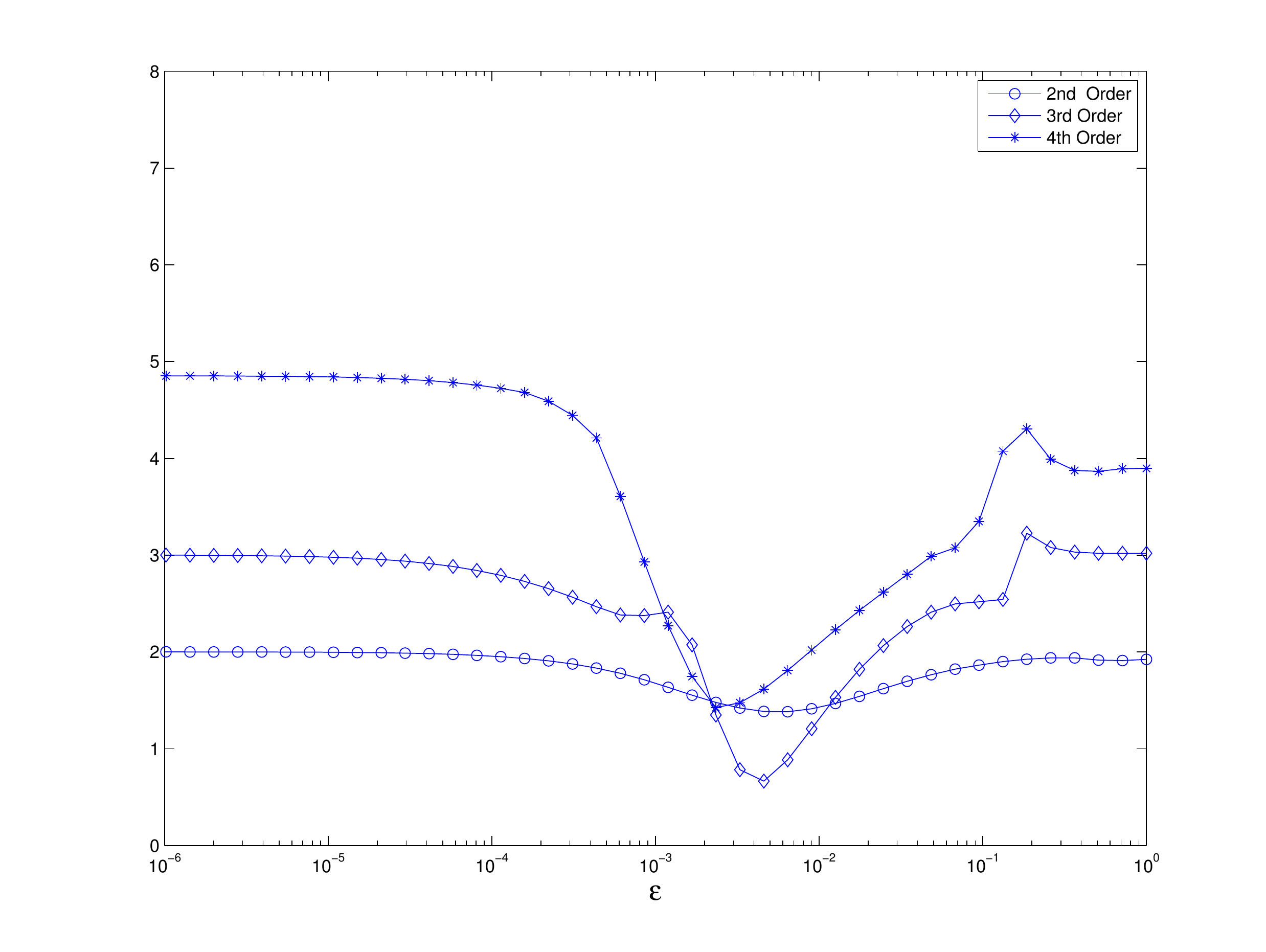}
\centering
\includegraphics[width=3.1in]{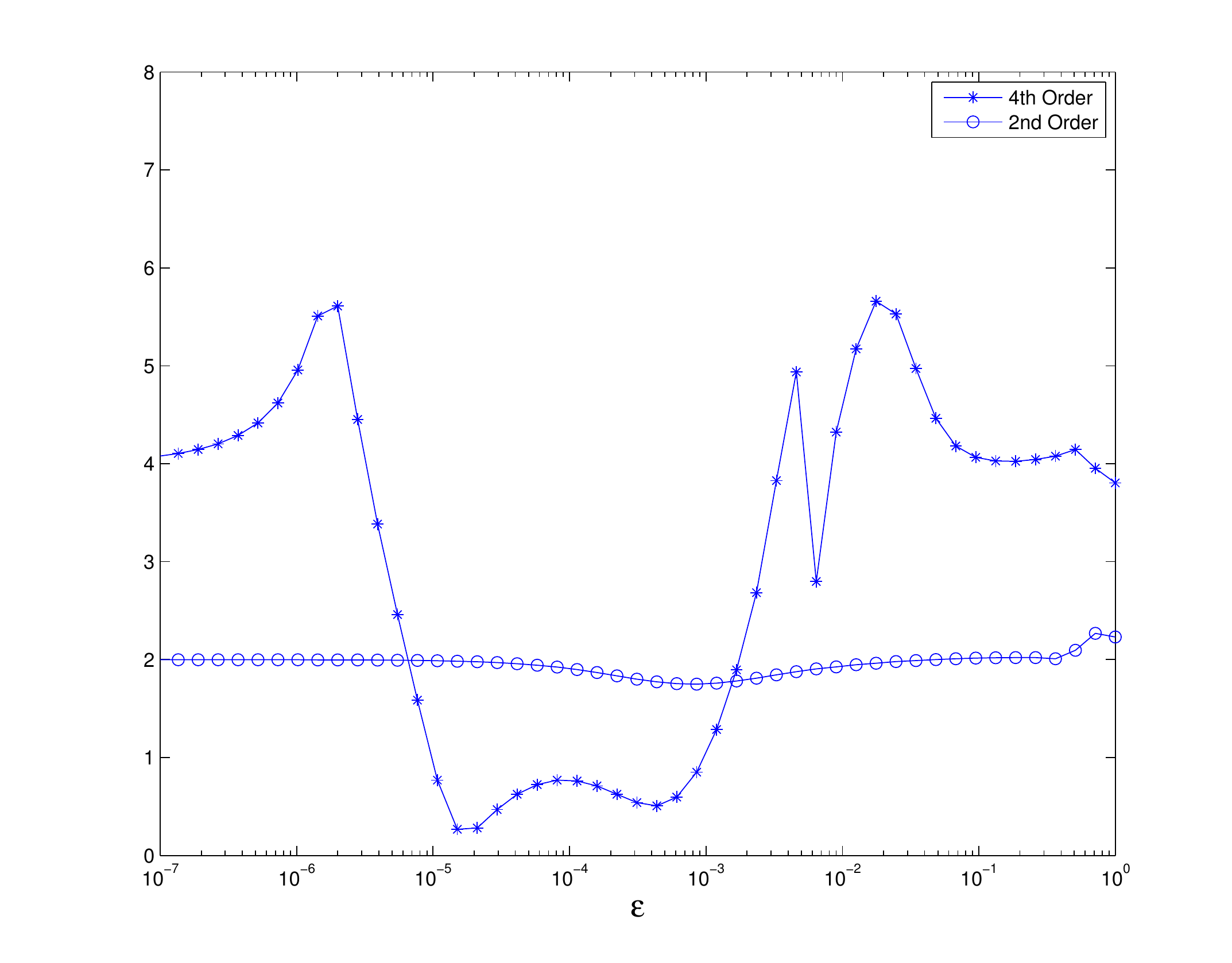}
\caption{Example \ref{HSwR}. Convergence rate for $u$-component versus $\varepsilon$. Left: for the method InDC-IMEX1-M-k, with $M=2, 3, 4$ corresponding to $k = 1, 2, 3$ corrections respectively. 
Right: second order IMEX2-ARS method GSA and InDC-IMEX2-ARS-GSA-M-k, with {$M=4$} and $k = 0, 1$ correction.
}
\label{fig5}
\end{figure}
\end{exa}
\begin{exa}\label{NHSwR}
Consider a nonlinear hyperbolic system with relaxation \cite{jin1995runge, boscarino2010class}
\begin{eqnarray} 
\label{Hyp2}
\begin{array}{l}
h_t + w_x  = 0, \\
w_t + (h+0.5 h^2)_x = -\frac{1}{\eps}(w - 0.5 h^2),
\end{array}  
\end{eqnarray}
where $\mathcal{U} = (h, w)$, $\mathcal{F}(\mathcal{U}) = (w, (h+0.5 h^2))^T$ and $\mathcal{G}(\mathcal{U}) = (0, -(w - 0.5 h^2))^T$. We consider a \emph{well-prepared} initial data given by $h(0,x) = 1 + 0.2\sin(8 \pi x)$ 
and $w = w_0 + \varepsilon w_1$ with $w_0 = f(h(0,x)) = 0.5h^2(0,x)$ and $w_1 = (f'(h(0,x)) - p'(h(0,x))\partial_xh(0,x)$
where $p(h) = (h + 0.5h^2)$. {The initial conditions are designed to be well-prepared in the same way as the previous example.}
The boundary condition is periodic. We evolve the solution to $T_{final} = 0.1$ before the shock forms. We perform our numerical test using a fixed space mesh with $\Delta x = 0.01$ for $x \in [0, 1]$. 

In Fig. \ref{fig6}, we plot the temporal convergence rate as a function of $\eps$, which is computed in the same fashion as those in Fig.~\ref{fig5}. Each point of the graph marked by a triangle in Figure \ref{fig6} is given by the expression 
\[
   {\rm order}_h(\eps) = \log_2(E_{h/2}(\eps)/E_h(\eps))
\]
where the errors $E_{h/2}$ and $E_h$ are computed for the same value of $\eps$. For example, the values along the dashed line in Figure \ref{fig6bis} are used to compute the point on Figure \ref{fig6} corresponding to $\eps=10^{-4}$. 
A similar behavior is observed as that in Fig.~\ref{fig5}: order of convergence strongly depends on $\varepsilon$. In particular, for small and large values of $\varepsilon$, the order of convergence is increased by InDC correction iterations, while for intermediate values of $\eps$, we observe that the order of convergence goes down to small values, even below 1. 
\begin{figure}
\centering
\includegraphics[width=3.0in]{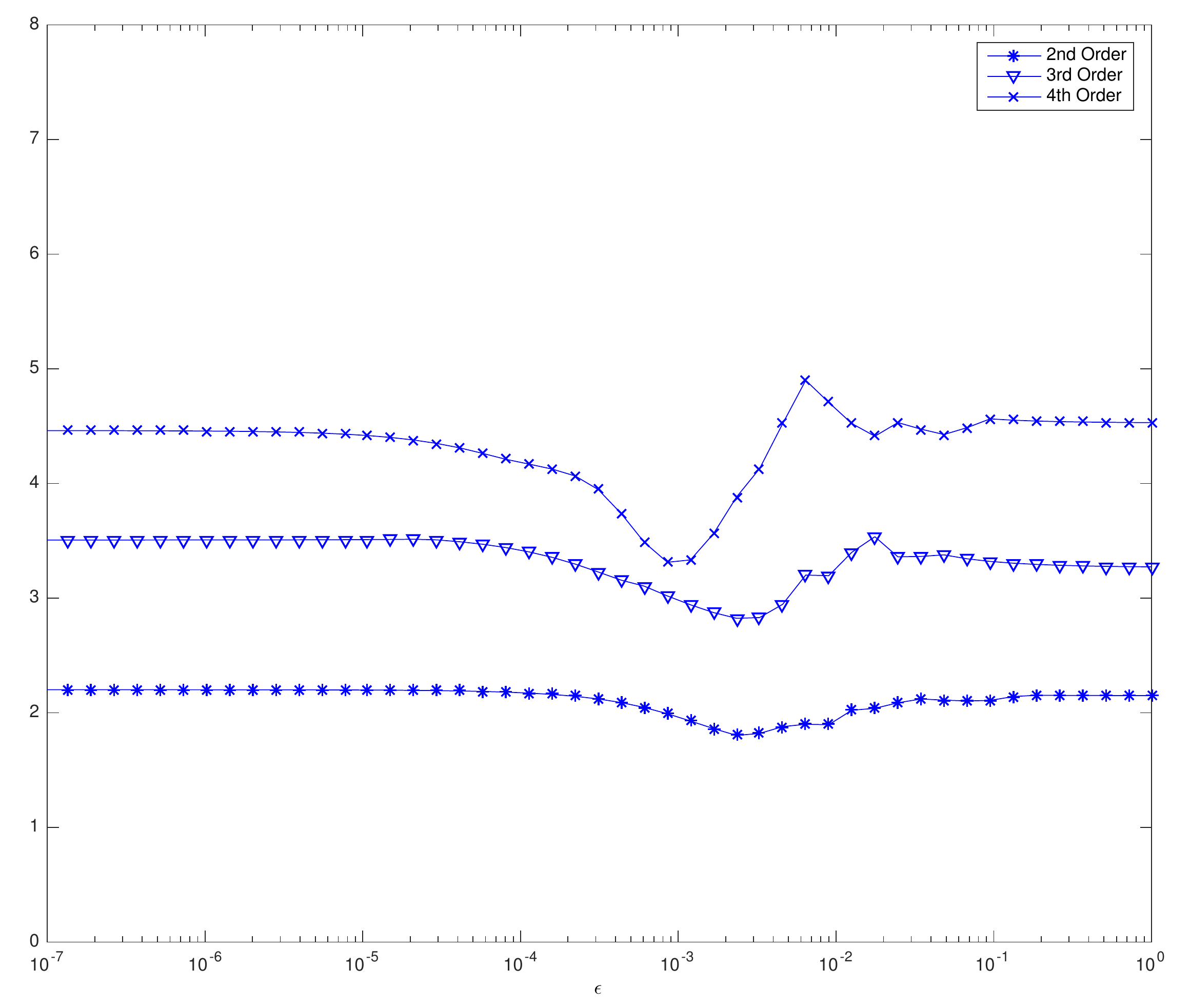}
\centering
\includegraphics[width=3.0in]{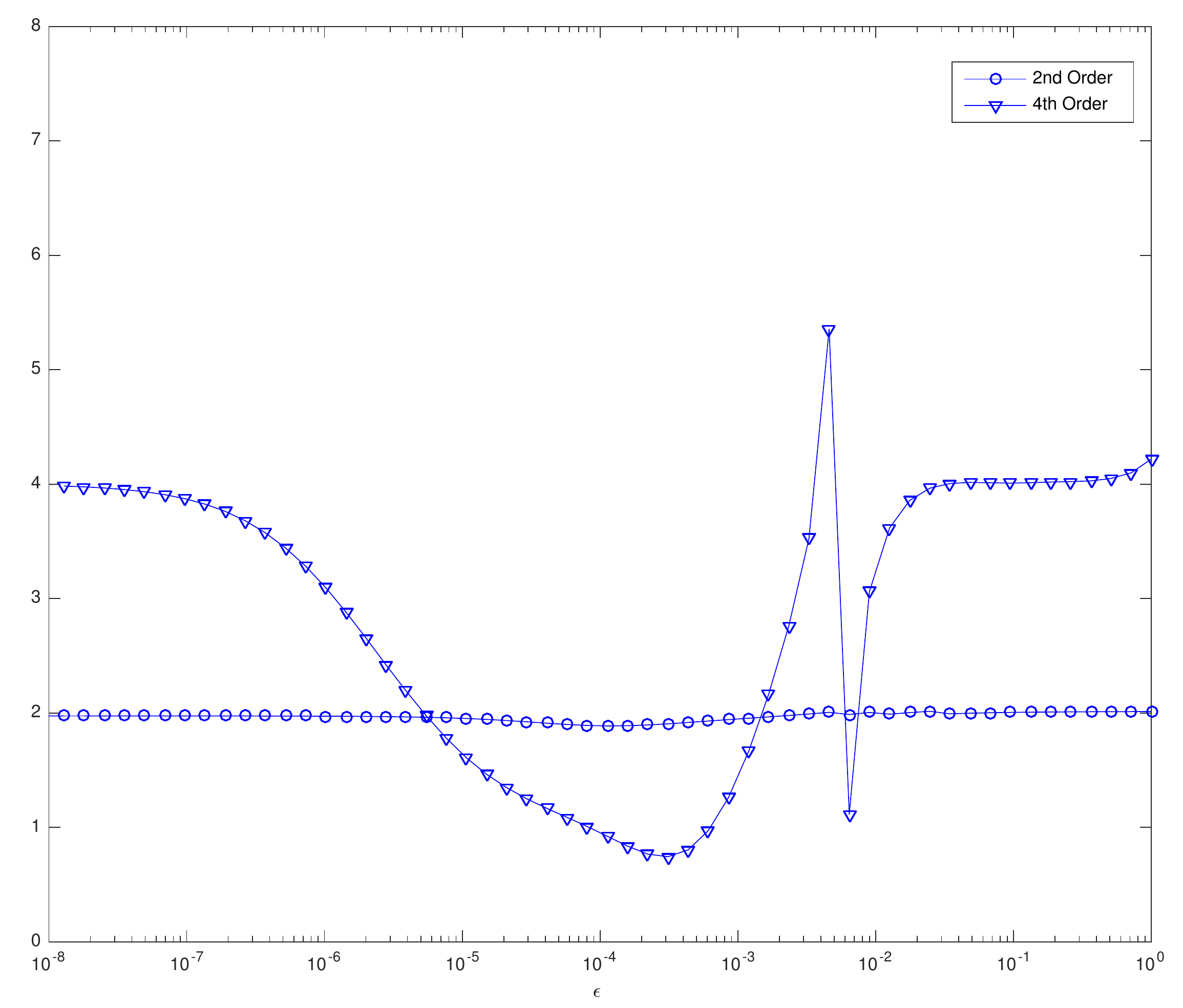}
\caption{Example \ref{NHSwR}. Convergence rate for $h$-component versus $\varepsilon$.
Left: for the method InDC-IMEX1-M-k, with $M=2, 3, 4$ corresponding to $k = 1, 2, 3$ corrections respectively. 
Right: second order IMEX2-ARS method GSA and InDC-IMEX2-ARS-GSA-M-k, with {$M=4$} and $k = 1$ correction.
}
\label{fig6}
\end{figure}

We show now that this observation is not in contradiction with the theoretical prediction. 
We start observing  that if $\eps \gg h$ then a classical analysis can be used to estimate the global error, i.e. the classical global error due to the presence of the small parameter $\eps$ is of the order $\mathcal{O}(h/\eps)^p$, with $p$ the order of the scheme. On the other hand, as a natural consequence of the Proposition \ref{classP},  if $\eps \ll h$, the global error satisfies (\ref{1_0est}).  As both estimates about the global error have to be satisfied (see Figure \ref{figErr}), we have: 
\begin{eqnarray}\label{estimateE}
\max_\eps err_\eps  \lesssim  \max_\eps \left( \min \left( \mathcal{O}\left(\frac{h}{\eps}\right)^p, \mathcal{O}(h^p) + \mathcal{O}(\eps h)\right)\right). 
\end{eqnarray}
Since the estimate is based on classical order (when $h \ll \eps$) and on the asymptotic expansion in $\eps$ (when $h \gg \eps$), 
we do not expect it to be sharp in intermediate regime, when $\eps \approx h$. 
Then a simple calculation shows that the uniform order is $\mathcal{O}(h^\frac{2p}{p+1})$, where the worst case takes place where $\mathcal{O}(\eps h) = \mathcal{O}((h/\eps)^p)$, i.e., for $\eps = \eps^*$ with
\begin{eqnarray}\label{eps_star}
\eps^* = \mathcal{O}(h^\frac{p-1}{p+1}).
\end{eqnarray}
Therefore we get 
$\max_\eps err_\eps =  \mathcal{O}(h^\frac{2p}{p+1})$.
This argument can be extended  to InDC IMEX R-K methods using the estimates (\ref{final_estimate}) given in the Theorem \ref{thm: IDC_IMEX_R-K}
In this case we get that the uniform order is 
\begin{eqnarray}\label{Uer2}
\max_\eps err_\eps = \mathcal{O}(h^\frac{2r}{r+1})
 \end{eqnarray}
 where $r = \min(M,s_k)$ is the classical order of the InDC IMEX R-K. 
 
{In order to show that our numerical observations in Fig.~\ref{fig6} are not in contradiction with the theoretical prediction eq.~(\ref{Uer2}) of the uniform order, we produce two error plots of the fourth order method (InDC-IMEX2-ARS-GSA-M-k, with $M = 4$ and $k = 1$) in Fig.~\ref{fig6bis}. 
On the left panel, for each curve and for each value of $\eps$, we plot the $L_1$ error, which is computed by comparing solutions for different values of time step, i.e. $E_i = \|S_{h/2^i}-S_{h/2^{i-1}}\|_{L^1}$, for $i = 1, 2, 3, 4$. The blue circles are the values corresponding to $\max_\varepsilon err_{\varepsilon}$. 
A lack of accuracy is observed in the intermediate regime when $h$ is approximately of the same order of $\eps$
On the right panel of Fig. \ref{fig6bis}, we act differently: for each $h$ we plot $\eps^* = arg (\max_\varepsilon err_{\varepsilon}(h))$ (identified by stars) and $err_{\varepsilon^*}$ (identify by circles) as a function of $\Delta t$. We compare it with the theoretical estimate (\ref{Uer2}). 
Using least square best fit on the four blue points, the computed uniform order appears to be $1.92$ while theoretical estimate (\ref{Uer2}) predicts a uniform order $1.6$ {(blue continuous line in Fig. \ref{fig6bis})} for $r =  4$. Likewise, a best fit for four red stars $\eps^{*} \propto h^{\alpha}$, gives $\alpha = 0.78$, while the theoretical prediction given in Eq. (\ref{eps_star}) gives $\alpha = (r-1)/(r+1) = 0.6$ {(red continuous line in Fig. \ref{fig6bis})} with $r = 4$. A similar behavior is obtained when adopting a third order scheme $r  = 3$ using a InDC-IMEX1- M-k, with $M = 3$ and $k =2$. The computed uniform order is approximately $1.9$, while theoretical prediction gives $1.5$. There is a small discrepancy between the computed values and the theoretical prediction. 
Such small discrepancy suggests that the estimate for the uniform order is not sharp. The reason is not fully understood and deserves deeper analysis.
}



\begin{figure}
\centering
\includegraphics[width=3.1in]{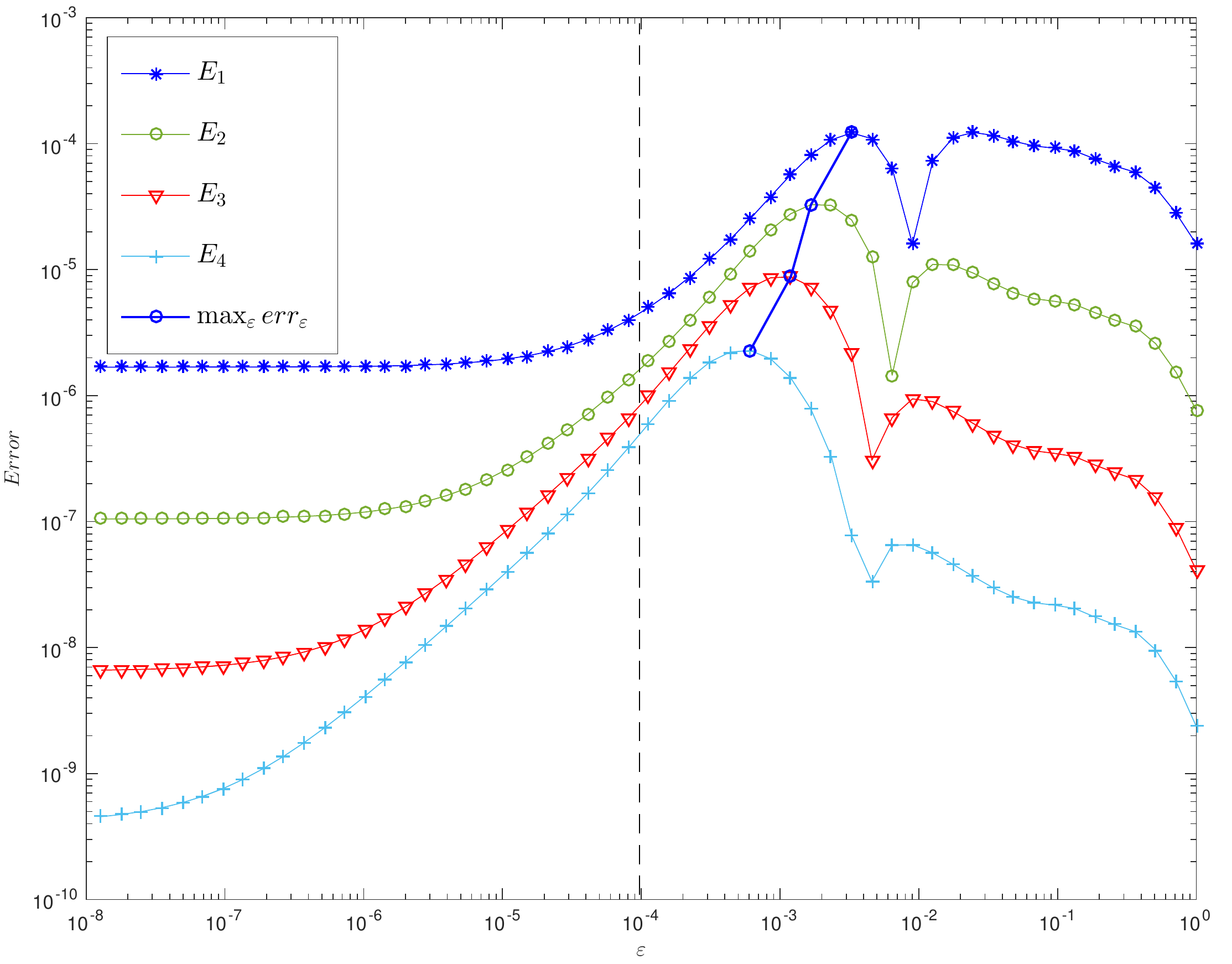}
\includegraphics[width=3.0in]{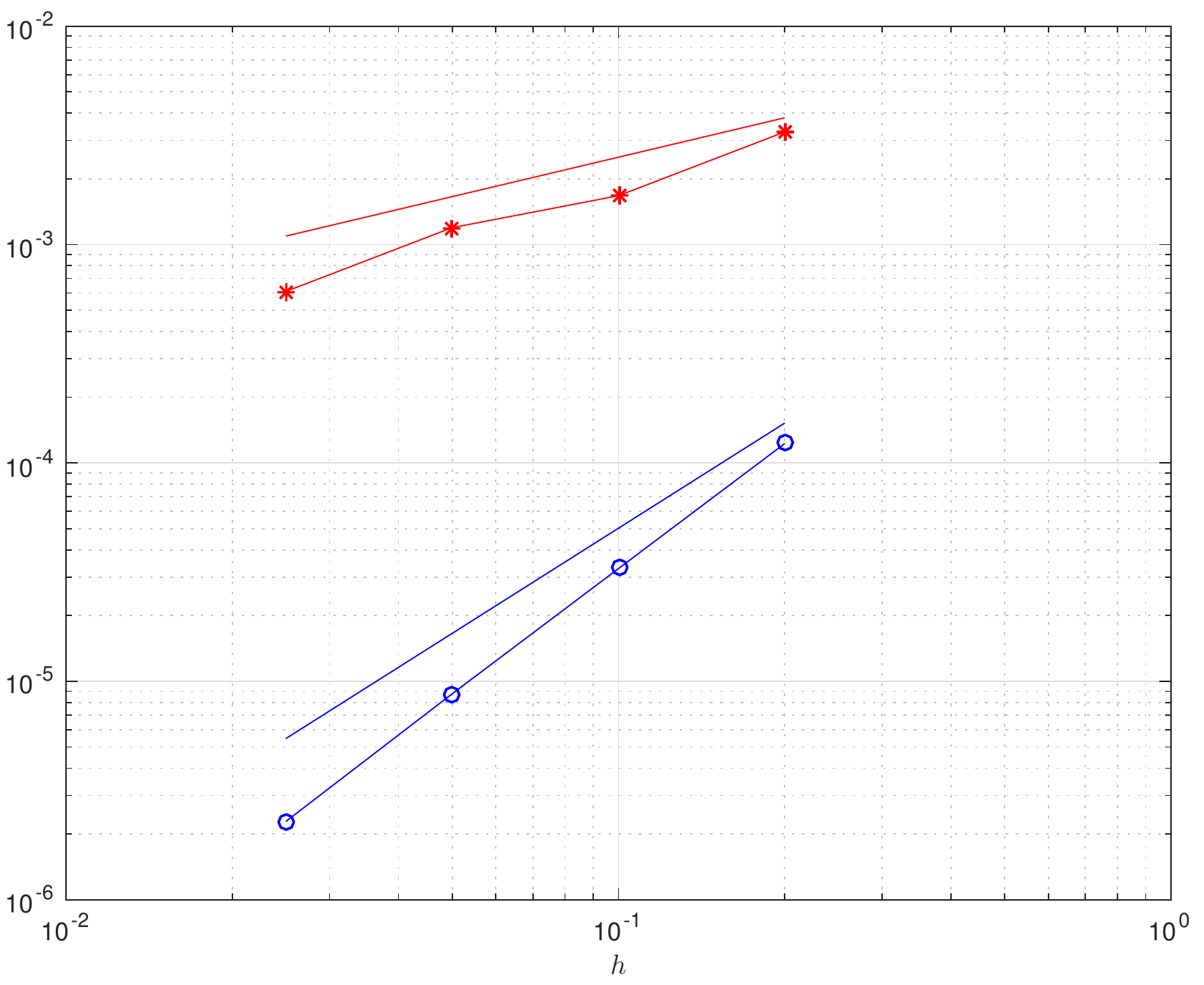}
\caption{Example \ref{NHSwR}. Logarithm scale for the error versus the parameter $\varepsilon$ with $10^{-8} \le \varepsilon \le 1$ (left panel) and 
$\eps^* = arg (\max_\varepsilon err_{\varepsilon})$ versus $h$ (identified by $-*$) and $err(\eps^*)$ versus $h$ (identified by "o") (right panel). The reference lines are linear curves with reference slopes $0.6$ and $1.6$ respectively. We used  InDC-IMEX2-ARS-GSA-M-k, with {$M=4$} and $k = 1$ correction.
}
\label{fig6bis}
\end{figure}
\begin{figure}
\centering
\includegraphics[width=5.0in]{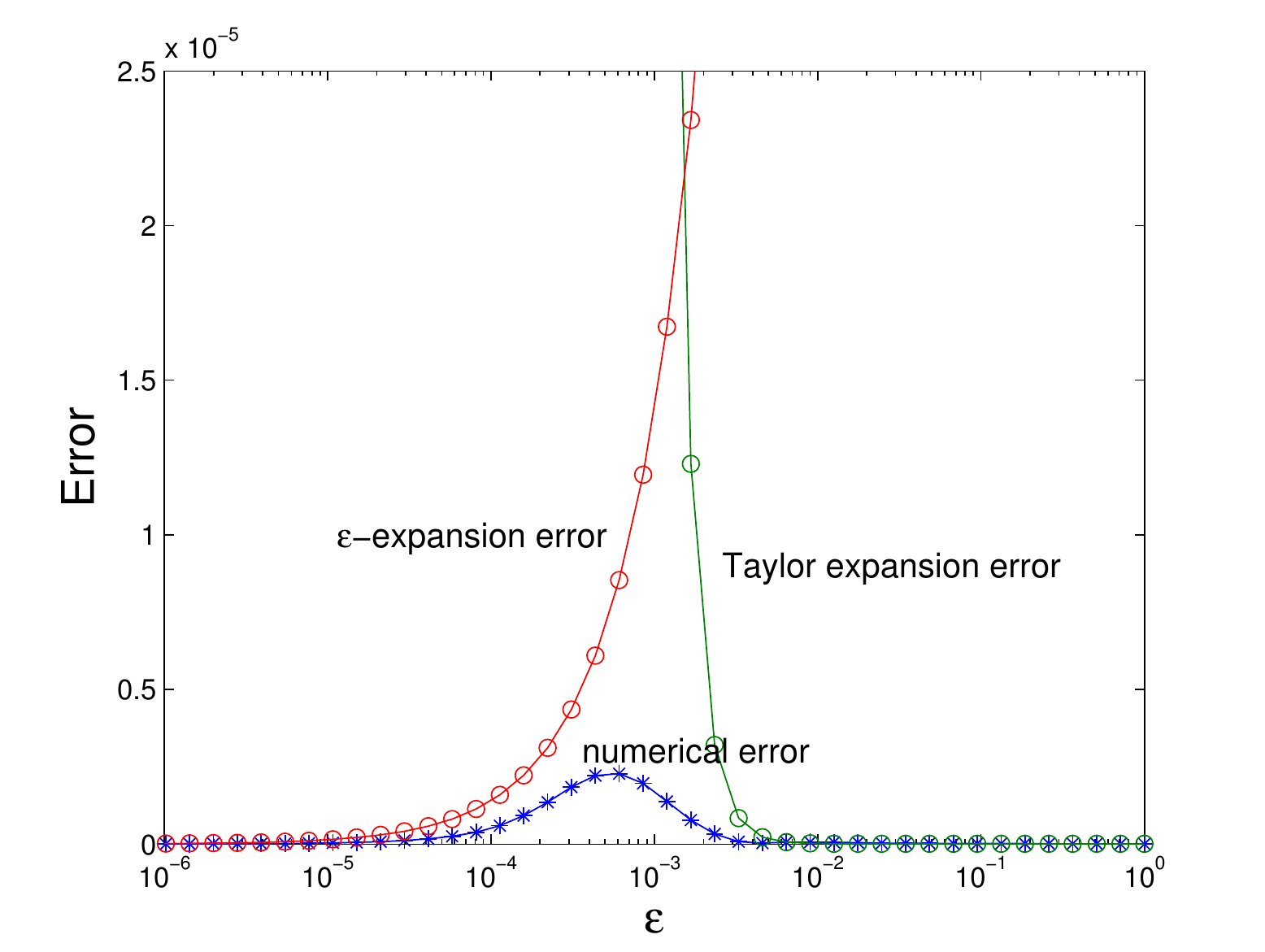}
\caption{
Numerical global error $(*)$  and theoretical global errors $(\circ)$ versus $\eps$.
In this Figure ``Taylor expansion error'' represents the global error given by the classical error analysis when $\eps \gg h$, and the ``$\eps$-expansion error'' the global one given by the asymptotic analysis when $\eps \ll h$ . The ``numerical error'' has been computed by a fourth order scheme, i.e. InDC-IMEX2-ARS-GSA-M-k, with M = 4 and k = 1 correction. This plot is consistent with the estimate (\ref{estimateE}) with $p = r = \min(M,s_k) = 4$.}
\label{figErr}
\end{figure}

\end{exa}

%% file: conclusion.tex
\section{Conclusions}
\label{sec6}
\setcounter{equation}{0}
\setcounter{figure}{0}
\setcounter{table}{0}

This paper studies the order of convergence of InDC-IMEX methods when applied to SSPs, using uniform distribution of quadrature points excluding the left-most point. 
Since InDC methods have a similar structure to R-K methods \cite{christlieb2009comments}, we construct the InDC-IMEX R-K methods as IMEX R-K methods with enlarged assembled Butcher tables and apply the convergence results in \cite{boscarino2008error} directly to the InDC-IMEX R-K methods.  Theoretical results on global error estimates in the form of $\varepsilon$-expansion are presented.
The InDC-IMEX schemes are applied to the classical Van der Pol equations and PDE systems in order to illustrate our theoretical findings. In particular the order reduction phenomenon is observed as expected for intermedia $\eps$, while high order convergent rates are observed in the asymptotic limit when $\varepsilon \to 0$ and when $\eps = \mathcal{O}(1)$ (similarly as in \cite{boscarino2010class}).

We also pointed out that  the globally stiffly accurate property of the IMEX R-K scheme is an important assumption: the fact that an IMEX RK is GSA guarantees that the assembled matrix of the corresponding InDC IMEX R-K scheme based on it is invertible, and this in turn implies that the high order accuracy in the limit as $\varepsilon\to 0$ is maintained. 
Note that the assumption of GSA provides a sufficient condition to guarantee the invertibility of the assembled matrix.
Although we do not know whether GSA property is necessary, we showed an example of InDC scheme based on a non GSA IMEX R-K, which lacks such asymptotic accuracy property.

Furthermore, we showed that even if InDC IMEX R-K can in principle be re-written as IMEX-RK with many stages, the actual implementation of such schemes as InDC IMEX R-K illustrated in the paper provides a systematic way to construct high order IMEX RK schemes, which would be far too complicated to write down as standard IMEX-RK schemes. 

%
%

%% file: acknowledgments.tex
\section*{Acknowledgments}
The research has been partially funded by 
ITN-ETN Marie-Curie Horizon 2020 program ModCompShock, {\em Modeling and computation of shocks and interfaces}, Project ID: 642768; 
by project F.I.R. 2014 {\em Charge transport in graphene and low dimensional systems}, University of Catania; and by INDAM-GNCS 2017 research project {\em Numerical methods for hyperbolic and kinetic equation and applications.}